\documentclass[11pt]{amsart}
\usepackage[utf8]{inputenc}
\usepackage{graphicx,color,verbatim}
\usepackage{amscd,amsmath,amsfonts,amssymb}
\usepackage{mathrsfs}
\usepackage{hyperref}
\usepackage{color,enumitem,graphicx}
 \usepackage[]{xcolor}

\newtheorem{theorem}{Theorem}
\newtheorem{lemma}[theorem]{Lemma}
\newtheorem{proposition}[theorem]{Proposition}
\newtheorem{remark}[theorem]{Remark}
\newtheorem{corollary}[theorem]{Corollary}

\numberwithin{equation}{section}

\title[Blow-up and local uniqueness]{Blow-up rate and local uniqueness for fractional Schr\"odinger equations with nearly critical growth}

\author[D.~Cassani]{Daniele Cassani$^\text{1}$}

\author[Y.~Wang]{Youjun Wang$^\text{2}$}
\address[D. Cassani]{\newline\indent Dip. di Scienza e Alta Tecnologia
	\newline\indent
	Universit\'{a} degli Studi dell'Insubria
	\newline\indent and
	\newline\indent RISM--Riemann International School of Mathematics
	\newline\indent Villa Toeplitz, Via G.B. Vico, 46 -- 21100 Varese}
    \email{\href{mailto:daniele.cassani@uninsubria.it}{daniele.cassani@uninsubria.it}}

\address[Y.~Wang]{\newline\indent Department of Mathematics
	\newline\indent
	South China University of Technology
	\newline\indent Guangzhou 510640, P. R. China
	}
\email{\href{mailto:scyjwang@scut.edu.cn}{scyjwang@scut.edu.cn}}

\thanks{(1) Corresponding author: daniele.cassani@uninsubria.it}

\subjclass[2010]{35A15; 35J60; 35B40}
\date{\today}
\keywords{Nonlocal equations, Fractional Laplacian, blow-up phenomena, ground states, critical growth.}

\begin{document}

\begin{abstract}
	We study quantitative aspects and concentration phenomena for ground states of the following nonlocal  Schr\"odinger equation
\begin{equation*}\label{maineq0}
 (-\Delta)^s u+V(x)u= u^{2_s^*-1-\epsilon}  \ \
     \text{in}\ \ \mathbb{R}^N,   \end{equation*}
   where $\epsilon>0$, $s\in (0,1)$, $2^*_s:=\frac{2N}{N-2s}$,  $N>4s$. We show that the ground state $u_\epsilon$ blows up and precisely with the following rate $\|u_\epsilon\|_{L^\infty(\mathbb{R}^N)}\sim   \epsilon^{-\frac{N-2s}{4s}}$, as
   $\epsilon\rightarrow 0^+$. We also localize the concentration points and, in the case of radial potentials $V$, we prove local uniqueness of sequences of ground states which exhibit a concentrating behavior.
   \end{abstract}

\maketitle

\section{Introduction}\label{introduction}

\noindent In this paper we consider the following class of nonlocal equations
\begin{equation}\label{maineq0}
 (-\Delta)^s u+V(x)u= u^{2_s^*-1-\epsilon}  \ \
       \text{in}\ \ \mathbb{R}^N,   \end{equation}
where $\epsilon\rightarrow 0^+$, $s\in (0,1)$, $2^*_s:=\frac{2N}{N-2s}$,  $N>4s$, $V:\mathbb{R}^N\rightarrow \mathbb{R}$ is a potential function   and
$$(-\Delta)^s u(x)=c_{N,s}\text{PV}\int_{\mathbb{R}^n}\frac{u(x)-u(y)}{|x-y|^{N+2s}}dy  $$   is the fractional Laplacian. Here, $c_{N,s}$ is a normalizing constant, PV stands for the Cauchy principal value.

\noindent For  fixed $\epsilon \in (0,2_s^*-2)$, under suitable conditions  on $V(x)$, it is known that equation (\ref{maineq0}) admits a positive  ground state $u_\epsilon$, see for instance \cite{MMF14, MFEV15,PEAJ12,AM16}. Moreover, if $V(x)=1$, then $u_\epsilon$ is spherically symmetric, see \cite{chelili2017a,DPV13}.
However, when  $\epsilon =0$  it follows from a Pohozaev type identity  that   (\ref{maineq0}) has no solutions  in $H^s(\mathbb{R}^N)\cap L^\infty(\mathbb{R}^N)$ if $V(x)+\frac{1}{2s}x\cdot \nabla V(x)\geq 0$ (and $\not\equiv 0$), see  Theorem \ref{Pohozae} or    \cite{XW13} in the special case $V(x)\equiv 1$. 

\noindent Therefore, it is natural to wonder what happens to the ground state $u_\epsilon$ as $\epsilon\rightarrow 0^+$. The main motivation of this paper is to achieve a better understanding of this phenomenon. This type of problems for semilinear equations, with the so-called nearly critical growth, were first studied in the unit ball of $\mathbb{R}^3$  by Atkinson and Peletier  in \cite{FL87} and then extended to  spherical
domains by  Brezis and Peletier  in \cite{HL89} and non-spherical domains by Han in \cite{Han91}. Indeed, they proved the solution $u_\epsilon$  blows up in the sense that $\|u_\epsilon\|_{L^\infty(\mathbb{R}^N)}\sim   \epsilon^{-\frac{1}{2}}$ as
   $\epsilon\rightarrow 0^+$.
More recently, their results were extended to
nonlocal problems   in bounded domains in \cite{WSK14}. Precisely,
 the authors in \cite{WSK14} study the following nonlocal problem
\begin{equation}\label{bounddomain}
\left\{\begin{array}{cccc}
 \mathcal{A}_su=u^{2_s^*-1-\epsilon} &\ \ &\text{ in }& \Omega,\\
u> 0&\ \ &\text{ in }&  \Omega,\\
u=0 &\ \ &\text{ on }&  \partial\Omega,
 \end{array}\right.
\end{equation}
where $\mathcal{A}_s$ denotes the fractional Laplace operator $(-\Delta)^s$  in $\Omega$ defined in terms of the spectrum of the Laplacian subject to Dirichlet boundary conditions. They proved that if $u_\epsilon$ is a solution to (\ref{bounddomain}) satisfying 
\begin{equation*}
\lim_{\epsilon\rightarrow 0}\frac{\displaystyle{\int_\Omega | \mathcal{A}_s^{\frac{1}{2}} u_\epsilon|^2dx}}{\|u_\epsilon\|^2_{2_s^*}}=S
\end{equation*} 
where $S$ is the best Sobolev constant in the embedding $H^s\hookrightarrow L^{s^*}$, then 
\begin{equation*}
\lim_{\epsilon\rightarrow 0} \epsilon \|u_\epsilon\|_{L^\infty(\Omega)}^2=b_{n,s}|\tau(x_0)|\, ,
\end{equation*}
where $b_{n,s}$ is a normalizing constant and $x_0\in \Omega$ is a  critical point of the  Robin function  $\tau(x)$. Besides, in \cite{SM16} the authors study the asymptotic behavior of solutions to the nonlocal nonlinear problem
\begin{equation}\label{bounddomain2}
\left\{\begin{array}{cccc}
(-\Delta_p)^su=|u|^{p_s^*-2-\epsilon}u &\ \ &\text{ in }& \Omega,\\
u=0 &\ \ &\text{ in }&   \mathbb{R}^N\setminus \Omega,
 \end{array}\right.
\end{equation}
where $p_s^*=\frac{Np}{N-sp}$, $N>ps$, $p>1$. They prove that ground state solutions
concentrate at a single point in $\overline{\Omega}$ and analyze the asymptotic behavior for
sequences of solutions at higher energy levels as $\epsilon\rightarrow 0$. In particular, in the semi-linear case $p=2$, they prove
that for smooth domains the concentration point   cannot lie on the boundary, and
identify its location in the case of annular domains.
  Regarding the nonlocal problem (\ref{bounddomain2}) for $p=2$,  we also refer to  \cite{PP14} for a profile decomposition approach and to \cite{PP15} for  $\Gamma$-convergence methods.

\noindent The purpose of this paper is twofold: on one side, under suitable conditions on $V(x)$,  we give a complete description of the blow up behavior of the  ground states of (\ref{maineq0}); on the other side, we
identify the location of the concentration points and then we establish local uniqueness of ground states.

\noindent Before stating our main results, let us make a few assumptions on $V(x)$. Throughout this paper, we assume  that      $V(x)$ satisfies the following conditions:

\begin{itemize}
\item[$(V_1)$] $V\in C^2$, $ 0<V_0\leq V(x)\leq V_\infty:=\sup\limits_{x\in \mathbb{R}^N}V(x)=\liminf\limits_{|x|\rightarrow +\infty}V(x)<+\infty$;
\item[$(V_2)$] The function $ x\cdot \nabla V(x)$   stays bounded in $\mathbb{R}^N.$
\end{itemize}

\noindent We consider here the fractional Sobolev space
$$H^s_{V}(\mathbb{R}^N):=\left\{u\in L^2(\mathbb{R}^N, V(x)\,dx):   [u]_s^2:=\int_{\mathbb{R}^N \times\mathbb{R}^N} \frac{|u(x)-u(y)|^2}{|x-y|^{N+2s}}dxdy<\infty  \right\}$$
endowed with the norm $$\|u\|_{s,V}:=\left( [u]_s^2+\int_{\mathbb{R}^N}V(x)u^2dx\right)^{\frac{1}{2}}.$$
Notice that  under $(V_1)$, $H^s(\mathbb{R}^N)$ which corresponds to the choice $V\equiv 1$ and $H^s_{V(x)}(\mathbb{R}^N)$ turn out to be equivalent in terms of norms as well as of elements.
Denote  by $D^s(\mathbb{R}^N)$ the closure of $C_0^\infty(\mathbb{R}^N)$ with respect to the norm $[u]_s$.  As usual $ \|\cdot\|_p$ denotes  $\|\cdot\|_{L^p(\mathbb{R}^N)}$   for $1\leq p\leq \infty$.

\noindent Let  \begin{equation}\label{fun0}
 S_{2_s^*-\epsilon}:=\inf_{ u\in H^s(\mathbb{R}^N) \setminus \{0\} }\frac{\|u\|_{s}^2}{\|u\|_{2_s^*-\epsilon}^2}=\inf_{ u\in H^s(\mathbb{R}^N)  }\{\|u\|_{s}^2:\|u\|_{2_s^*-\epsilon}^{ 2_s^*-\epsilon}=1\}.
\end{equation}
By Lion’s concentration compactness, minimizers
for ${S}_{2_s^*-\epsilon}$ always exist    and one may assume they do not change sign, \cite{PEAJ12, DPV13},  Moreover, they are  radially symmetric, see  \cite{FRLE}.  Here, we will consider only positive minimizers.

\noindent Recall also the Sobolev constant
 \begin{equation}\label{bestconst}S:=\inf_{ u\in D^{s}(\mathbb{R}^N) \setminus \{0\}} \frac {[u]_{s}^2}{ \|u\|_{2_s^*}^2}.\end{equation}


\noindent For each fixed $\epsilon \in (0,2_s^*-2)$, let $w_\epsilon$ be a positive minimizer for

 \begin{equation}\label{fun00}
 S_{2_s^*-\epsilon}^V=\inf_{ u\in H_{V}^s(\mathbb{R}^N)  }\{\|u\|_{s,V}^2:\|u\|_{2_s^*-\epsilon}^{2_s^*-\epsilon}=1\},
\end{equation}
or  equivalently,  \begin{equation}\label{fun000}
 S_{2_s^*-\epsilon}^V=\inf_{ u\in H_{V}^s(\mathbb{R}^N) \setminus \{0\} }I_\epsilon(u),
\end{equation}
where $$I_\epsilon(u)=\frac {\|u\|_{s,V}^2}{ \|u\|_{2_s^*-\epsilon }^2}.$$

\medskip

\noindent Our main results are the following:

\begin{theorem} \label{th2} Assume $(V_1),$ $(V_2)$,  $N>4s$ and that $u_\epsilon$ is a ground state of (\ref{maineq0}), namely satisfying (\ref{fun000}), which has a maximum point $x_\epsilon$ such that $x_\epsilon\rightarrow x_0$ as $\epsilon\rightarrow 0^+$. Then,
\begin{itemize}
\item[$(1)$] $\lim\limits_{\epsilon\rightarrow 0^+}S_{2_s^*-\epsilon}^V= S;$
 \item[$(2)$]  $$\lim_{\epsilon\rightarrow 0^+} \epsilon \|u_\epsilon\|_{\infty }^{\frac{4s}{N-2s}}=A_{N,s}\left[V(x_0)+\frac{1}{2s}x_0\cdot \nabla V(x_0)\right]$$
\end{itemize}
where $$A_{N,s}:= \frac{2^{2(N+1)}N^2\pi^{\frac{N}{2}} \Gamma\left(\frac{N-4s}{2}\right)}{(N-2s)^2\Gamma(N-2s)} S^{-\frac{N}{2s}} .$$
\end{theorem}
\noindent In particular we have 
\begin{corollary} \label{th1} Assume $N>4s$ and that $u_\epsilon$ is  a minimizer for (\ref{fun0}). Then, we have:
\begin{itemize}
\item[$(1)$] $\lim\limits_{\epsilon\rightarrow 0^+} S_{2_s^*-\epsilon}= S;$
\item[$(2)$]  $$\lim_{\epsilon\rightarrow 0^+} \epsilon \|u_\epsilon\|_{\infty }^{\frac{4s}{N-2s}}=A_{N,s}\ .$$
\end{itemize}
\end{corollary}

\noindent Notice that in Theorem \ref{th2}, we assume that the maximum point $x_\epsilon$ does converge. However, under conditions $(V_1)$ and $(V_2)$,  one of the main difficulties is that $x_\epsilon$ may actually escape  to infinity as $\epsilon\rightarrow 0^+$.
  In what follows, we prove  that if $N>6s$, then the maximum point $x_\epsilon$ must be bounded, and
therefore converging, up to a subsequence,  to  a global minimum point of $V(x)$ provided $\inf_{x\in \mathbb{R}^N} V(x)<V_\infty$. More precisely, we have the following

\begin{theorem}  \label{th3} Assume $(V_1),$ $(V_2)$ with  $\inf_{x\in \mathbb{R}^N} V(x)<V_\infty$,   $N>6s$ and that $u_\epsilon$ is a ground state of (\ref{maineq0})  (in the sense of (\ref{fun000})) which has a maximum point $x_\epsilon$. Then, there exists a subsequence   $\{x_{\epsilon_j}\}$ of $\{x_{\epsilon}\}$ such that:
\begin{itemize}
\item[$(1)$] $\lim\limits_{ j\rightarrow \infty}x_{\epsilon_j}= x_0$,  where $x_0$ is a global minimum point of $V(x)$;
\item[$(2)$] $\lim\limits_{ j\rightarrow \infty}S_{2_s^*-\epsilon_j}^V= S;$
\item[$(3)$]  $$\lim_{j\rightarrow \infty} \epsilon_j \|u_{\epsilon_j}\|_{\infty}^{\frac{4s}{N-2s}}=A_{N,s} V(x_0).$$
\end{itemize}
\end{theorem}

\noindent What stated in Theorem \ref{th3} opens a natural question: \textit{is there more than one blow-up ground state  sequence such that the maxima concentrate at  the same point?}

\noindent We do not have a full answer, however let us consider a special case. Assume $V(x)$ is radial and that there exist two radial  ground state sequences $u_{\epsilon_j}^1$ and  $u_{\epsilon_j}^2$     of (\ref{maineq0}) such that $\|u_j^i\|_{\infty}=u_j^i(0),$  $i=1,2$.  Set $\mu_j^i:=\|u_j^i\|_{\infty}^{-\frac{2_s^*-2-\epsilon_j}{2s}} $, $i=1,2$.
We have the following local uniqueness result.

\begin{theorem}  \label{th4} Assume $(V_1),$ $(V_2)$, that $V(x)=V(|x|)$ is radial,   $N>4s$ and there exist two radial  ground state sequences   $u_{\epsilon_j}^1$ and  $u_{\epsilon_j}^2$    of (\ref{maineq0})  satisfying $\|u_j^i\|_{\infty}=u_j^i(0),$  $i=1,2$,  $(1)$, $(2)$ and $(3)$ of Theorem \ref{th3}. Then, there exists $\epsilon_0>0$ such that for any $\epsilon_j\in (0,\epsilon_0)$, we have $u_{\epsilon_j}^1=u_{\epsilon_j}^2$, provided $\mu_j^1=\mu_j^2$. More precisely, up to rescaling, we have the following local uniqueness result 
$$u_{\epsilon_j}^2(x)=\left(\frac{\mu_j^1}{\mu_j^2}\right)^{\frac{2s}{2_s^*-2-\epsilon_j}}u_{\epsilon_j}^1\left(\frac{\mu_j^1}{\mu_j^2} x\right)$$
or equivalently
$$u_{\epsilon_j}^2(x)=\frac{\|u_{\epsilon_j}^2\|_{\infty}}{\|u_{\epsilon_j}^1\|_{\infty}}u_{\epsilon_j}^1
\left[\left(\frac{\|u_{\epsilon_j}^2\|_{\infty}}{\|u_{\epsilon_j}^1\|_{\infty}}\right)^{\frac{2_s^*-2-\epsilon_j}{2s}} x\right].$$
\end{theorem}

\begin{remark}
In Theorems \ref{th4} we assume that $\|u_j^i\|_{\infty}=u_j^i(0),$  $i=1,2$. Indeed, the results are still true if $\|u_j^i\|_{\infty}=u_j^i(x_j)$ for some $x_j\in \mathbb{R}^N$ and $\lim_{j\rightarrow \infty} x_j=x_0$.  Let $w_j^i(x)=u_j^i(x+x_j)$, then $w_j^i$ satisfies
$$ (-\Delta)^s w_j^i+V(x+x_j)w_j^i= (w_j^i)^{2_s^*-1-\epsilon_j}  \ \
       \text{in}\ \ \mathbb{R}^N.$$
Define
$$v_j^i(x)=\mu_j^{\frac{2s}{2_s^*-2-\epsilon_j}}w_j^i(\mu_j^ix).$$
Then $0<v_j^i(x)\leq 1$, $v_j^i(0)=1$, and satisfies
\begin{equation}\label{res1}
   \begin{split}
 (-\Delta)^s v_j^i+(\mu_j^i)^{2s}V(x_j+\mu_j^i x)v_j^i=(v_j^i)^{2_s^*-1-\epsilon_j} \ \ \text{in} \ \ \mathbb{R}^N.
 \end{split}
 \end{equation}

\end{remark}

\subsection*{Overview} The asymptotic behavior of ground states to nonlocal problems has  attracted remarkable attention in recent years.  In \cite{MFEV15}, the authors studied    the singularly perturbed fractional Schr\"{o}dinger equation
\begin{equation}\label{Sc}
 \epsilon^{2s}(-\Delta)^s u+V(x)u= u^{p} \ \  \text{in}\ \ \mathbb{R}^N,    \end{equation}
where $1<p<2_s^*-1$.
They proved that concentration points turn out to be critical points for $V$. Moreover, they proved that if the potential $V$ is coercive and has a unique
global minimum, then ground states concentrate at that minimum point as $\epsilon\rightarrow 0$.
In \cite{JMY14}, by means of a Lyapunov-Schmidt reduction method, the authors proved the existence of various type of concentrating solutions, such as multiple spikes and clusters, such that each of the local maxima converge to a critical point of $V$ as $\epsilon\rightarrow 0$, see also \cite{ALMI16,doms14}. In \cite{BMU191},  the authors considered the nonlocal scalar field equation
\begin{equation}\label{SFE}
(-\Delta)^s u+\epsilon u= |u|^{p-2}u-|u|^{q-2}u \ \  \text{in}\ \ \mathbb{R}^N,    \end{equation}
where $2<p<q$. For $\epsilon$ small, they proved the existence and qualitative properties of positive solutions when $p$ is subcritical, supercritical or critical Sobolev exponent.  For the existence of positive solutions of nonlocal equations with a small parameter see also \cite{BM19,DMPV17}.

\noindent Loosely speaking, all the results mentioned above were concerned with the characterization of concentration of  ground states. The purpose of this paper is quite different as we focus on quantitative aspects of concentrating solutions . Let us emphasize that Theorem \ref{th3} can be seen as a nonlocal analog of the results in \cite{PANWANG92,WANG96}. In \cite{PANWANG92}, the authors studied
the behavior of the ground states of equation
\begin{equation}\label{SCH}
  -\Delta  u +K(x)u= u^{2^*-1-\epsilon} \ \ \text{in}\ \ \mathbb{R}^N.
   \end{equation}
Under some geometric assumptions on  $K(x)$, they proved the existence of ground states $u_\epsilon$. Moreover,    the maximum point $x_\epsilon$ of   $u_\varepsilon$ is bounded and $\|u\|_{L^\infty(\mathbb{R}^N)}\sim \epsilon^{-\frac{N-2}{4}}$ as $\epsilon\rightarrow 0^+$.
In \cite{WANG96}, the author further identified  the location of the blow-up point. In the present paper,  though conditions $(V_1)$ and $(V_2)$  guarantee the existence of the ground state solution $u_\epsilon$, it is not true in general that the maximum point $x_\epsilon$ of $u_\epsilon$ stays bounded as $\epsilon\rightarrow 0^+$ and this yields a major difficulty. 

\medskip

\noindent The paper is organized as follows: existence of minimizers, local boundedness estimates of solutions and a Pohozaev type identity are established in the preliminary Section \ref{pre}. In Section  \ref{S2}, we study the asymptotic behavior of ground states, including a uniform bound up to rescaling. Section \ref{S3} is devoted to identify the location of blow-up points, whence in Section \ref{S4} we prove the local uniqueness of ground states. 

\noindent Throughout this paper, $C$ will denote a positive constant which may vary from line to line.

\section{Preliminaries}\label{pre}

\noindent Here for the convenience of the reader we prove some auxiliary results. Consider first the following constrained minimization: 
\begin{equation}\label{fun}
 S_{2_s^*-\epsilon}^V:=\inf_{ u\in H_{V}^s(\mathbb{R}^N)  }\{\|u\|_{s,V}^2:\|u\|_{2_s^*-\epsilon}^{2_s^*-\epsilon}=1\}.
\end{equation}
In the special case $V(x)=1$,  minimizers
for $S_{2_s^*-\epsilon}$ always exist    and do not change sign, see e.g. \cite{PEAJ12, DPV13},  Moreover, they are  radially symmetric, see  \cite{FRLE}.

\begin{theorem} \label{ThA} Assume $(V_1)$ holds. Then,  $S_{2_s^*-\epsilon}^V$ is  achieved at some $w_\epsilon\in   H_{V}^s(\mathbb{R}^N).$
\end{theorem}
\begin{proof} We assume that $V(x)\not\equiv V_\infty$, otherwise the result is obvious.
 Let $\{w_n\}$ be a minimizing sequence for $ S_{2_s^*-\epsilon}^V$. Since $|w_n|\in  H_{V}^s(\mathbb{R}^N)$ and $[|w_n|]_s\leq [w_n]_s$, we may assume that $w_n$ is nonnegative.  Clearly, $\{w_n\}$ is bounded in $ H_{V}^s(\mathbb{R}^N)$  and $\|w_n\|_{2_s^*-\epsilon}^{2_s^*-\epsilon}=1.$
Therefore, up to subsequences if necessary, there exists $w\in H_{V}^s(\mathbb{R}^N)$  such that $w_n\rightharpoonup w$ in  $H_{V}^s(\mathbb{R}^N)$ as $n\rightarrow +\infty$.
Let
$\ell=\|w\|_{2_s^*-\epsilon}^{2_s^*-\epsilon}$, then $0\leq \ell \leq 1$.  We next claim that actually $ \ell =1$.

\noindent Indeed, let 
\begin{equation}\label{fun}
 S_{2_s^*-\epsilon}^\infty:=\inf_{ u\in H_{V_{\infty}}^s(\mathbb{R}^N)  }\{\|u\|_{s,V_\infty}^2:\|u\|_{2_s^*-\epsilon}^{2_s^*-\epsilon}=1\}.
\end{equation}
Then minimizer $u$  for  $S_{2_s^*-\epsilon}^\infty$ exists and does not change sign,  see e.g. \cite{PEAJ12, DPV13}. Without loss of generality, we assume $u$ is positive.
  Using this $u$ as a test function we can show that if $V (x)$ is not identically equal to $V_\infty$, then $S_{2_s^*-\epsilon}^V<S_{2_s^*-\epsilon}^\infty$.

\noindent Set $v_n=w_n-w$, then  $v_n\rightharpoonup 0$ in  $H_{V}^s(\mathbb{R}^N)$ and $v_n\rightarrow 0$ in $L_{loc}^2(\mathbb{R}^N)$   as $n\rightarrow +\infty$ and by   Bresiz-Lieb lemma, we have
 \begin{equation}\label{exis5}
   \begin{split}\lim_{n\rightarrow+\infty}\|v_n\|_{2_s^*-\epsilon}^{2_s^*-\epsilon}=&\lim_{n\rightarrow+\infty}\|w_n\|_{2_s^*-\epsilon}^{2_s^*-\epsilon}- \|w\|_{2_s^*-\epsilon}^{2_s^*-\epsilon} \\
  =&1- \|w\|_{2_s^*-\epsilon}^{2_s^*-\epsilon}\\
  =&1-\ell.
   \end{split}
 \end{equation}
On the one hand we have
 \begin{equation}\label{exis6}
   \begin{split}
  \lim_{n\rightarrow +\infty} \|w_n\|_{s,V}^2=&  \lim_{n\rightarrow +\infty}\|v_n\|_{s,V}^2 +\|w\|_{s,V}^2\\&+2 \lim_{n\rightarrow +\infty}\int_{\mathbb{R}^N\times\mathbb{R}^N}\frac{(v_n(x)-v_n(y))(w(x)-w(y))}{|x-y|^{N+2s}}dxdy\\&
   + \lim_{n\rightarrow +\infty}\int_{\mathbb{R}^N}V(x)v_nwdx\\
   =&  \lim_{n\rightarrow +\infty}\|v_n\|_{s,V}^2 +\|w\|_{s,V}^2.
 \end{split}
 \end{equation}

 \noindent On the other hand,  by $(V_1)$  and   $v_n\rightarrow 0$ in $L_{loc}^2(\mathbb{R}^N)$  as $n\rightarrow +\infty$, we have
  $$\lim_{n\rightarrow +\infty}\int_{\mathbb{R}^N}[V(x)-V_\infty]w_n^2dx=0.$$
Thus,
 we have\begin{equation}\label{exis7}
   \begin{split}
 \lim_{n\rightarrow +\infty}\|v_n\|_{s,V(x)}^2
  = \lim_{n\rightarrow +\infty}\|v_n\|_{s,V_\infty}^2.
 \end{split}
 \end{equation}

\noindent For $0\leq \ell\leq 1$, by the definitions of $S_{2_s^*-\epsilon}^V$ and $S_{2_s^*-\epsilon}^\infty$,   we have
\begin{equation}\label{exis8}
   \begin{split}\|w\|_{s,V(x)}^2\geq \ell^{\frac{2}{2_s^*-\epsilon}}S_{2_s^*-\epsilon}^V  \end{split}
 \end{equation} and by (\ref{exis5}), we get
 \begin{equation}\label{exis9}
   \begin{split}  \lim_{n\rightarrow +\infty} \|v_n\|_{s,V_\infty}^2\geq  \lim_{n\rightarrow +\infty}\|v_n\|_{2_s^*-\epsilon}^2 S_{2_s^*-\epsilon}^\infty =(1-\ell)^{\frac{2}{2_s^*-\epsilon}}S_{2_s^*-\epsilon}^\infty. \end{split}
 \end{equation}

 \noindent Therefore,  by (\ref{exis6}), (\ref{exis7}), (\ref{exis8}) and (\ref{exis9}), we have
  \begin{equation}\label{exis10}
   \begin{split}
 S_{2_s^*-\epsilon}^V \geq\ell^{\frac{2}{2_s^*-\epsilon}}S_{2_s^*-\epsilon}^V+   (1-\ell)^{\frac{2}{2_s^*-\epsilon}}S_{2_s^*-\epsilon}^\infty
 \end{split}
 \end{equation}
 which gives
   \begin{equation}\label{exis11}
   \begin{split}
1-\ell^{\frac{2}{2_s^*-\epsilon}}\geq  (1-\ell)^{\frac{2}{2_s^*-\epsilon}}.
 \end{split}
 \end{equation}
Thus, from  (\ref{exis11}), we deduce that $\ell=0$ or $\ell=1$. If $\ell=0$, then from (\ref{exis10}), we get $ S_{2_s^*-\epsilon}^V \geq S_{2_s^*-\epsilon}^\infty$, which is a contradiction. Thus, $\ell=1$, that is, $\|w\|_{2_s^*-\epsilon}=1$ and thus $w$ is a minimizer of $S_{2_s^*-\epsilon}^V$.
\end{proof}

\begin{remark} Notice that in the proof of Theorem \ref{ThA},   condition $S_{2_s^*-\epsilon}^V<S_{2_s^*-\epsilon}^\infty$ plays an important role. This is guaranteed by condition $(V_1)$ with $V(x)\not\equiv V_\infty$.

\noindent By the Lagrange multiplier rule, there exists some $\lambda_\epsilon>0$ such that  $w_\epsilon$ is a solution of the following equation
\begin{equation}\label{maineq01-1}(-\Delta)^s u+V(x)u= \lambda_\epsilon u^{2_s^*-1-\epsilon}  \ \ \text{in}\ \ \mathbb{R}^N.\end{equation}
By the maximum principle $w_\epsilon>0$. In fact, $w_\epsilon\geq0$, and if there exists some $x_0$ such that $w_\epsilon(x_0)=0$, then
\begin{equation}\label{maineq01-1}0\leq (-\Delta)^s w_\epsilon (x_0)+V(x)w_\epsilon(x_0)=c_{N,s}\text{PV}\int_{\mathbb{R}^n}\frac{-u(y)}{|x_0-y|^{N+2s}}dy<0,\end{equation}
thus a contradiction.

\end{remark}

\begin{remark} If $V(x)$ is radial, by means of symmetric rearrangement techniques, we may
assume that $w_n$ is radially symmetric  (cf. \cite{PAR11}). Thus, the minimizer  $w$ is radial.

\end{remark}

\noindent Next we proof a Pohozaev type identity for  the  nonlocal equation
\begin{equation}\label{Poh0}(-\Delta)^s u =f(x,u) \ \ \text{in}\ \   \mathbb{R}^N. \end{equation}
The argument is similar to \cite{BMU17}, where the  Pohozaev identity for   autonomous nonlocal equations was established, hence we just stress the differences. 
\begin{theorem}[Pohozaev identity] \label{Pohozae}
Let $u\in H^s(\mathbb{R}^N)\cap L^\infty(\mathbb{R}^N)$ be a positive solution to (\ref{Poh0})
 and $F(x,t)\in L^1(\mathbb{R}^N)$, where $F(x,t)=\int_0^tf(x,s)ds$. Then we have
\begin{equation}\label{Poh}
   \begin{split}
\frac{ N-2s}{2} \int_{ \mathbb{R}^N}  f(x,u)u dx=\int_{ \mathbb{R}^N}  \left[ NF(x,u)    + (x\cdot \nabla_x F(x,u) )\right] dx.
 \end{split}
 \end{equation}

\end{theorem}

\begin{proof}

Let $u$ be a bounded weak nontrivial solution. Suppose that $w$ is the harmonic extension of $u$, see e.g. \cite{CASI07}. Then, $w$ satisfies
\begin{equation}\label{Poh1}
\left\{\begin{array}{cccc}
-\text{div}(y^{1-2s}\nabla w)=0&\ \ &\text{ in }&   \mathbb{R}^{N+1}_+,\\
\frac{\partial w}{\partial\nu^{s}}= f((\cdot,0), w(\cdot,0))&\ \ &\text{ in }&   \mathbb{R}^N\times \{y=0\}. \end{array}\right.
\end{equation}

\noindent For $r>0$, let
$$B_r:=\{(x,y)\in \mathbb{R}^{N+1}:|(x,y)|<r\}$$ and $$B_r^+=B_r\cap  \mathbb{R}^{N+1}_+,\ \
Q_r=B_r^+\cup (B_r\cap(\mathbb{R}^N\times \{0\})).$$

\noindent Let $\phi\in C_0^\infty(\mathbb{R}^{N+1})$ with $0\leq \phi \leq 1$, $\phi=1$ in $B_1$ and $\phi=0$ in $B_2^c$, $|\nabla \phi|\leq 2$.   For $R>0$, let
$$\varphi_R(x,y)=\varphi\left(\frac{(x,y)}{R}\right),$$ where $\varphi:=\phi|_{\mathbb{R}^{N+1}_+}$.

\noindent Then, multiplying  (\ref{Poh1}) by $((x, y)
\cdot \nabla w)\varphi_R$ and integrating in $\mathbb{R}^{N+1}_+$, we have,
\begin{equation}\label{PO1}
   \begin{split}
 \int_{Q_{2r}}\text{div} (y^{1-2s}\nabla w)[((x,y)\cdot \nabla w)\varphi_R]dxdy=0.
 \end{split}
 \end{equation}
From (\ref{PO1}), by  integrating by parts,  we get
\begin{equation}\label{Poh2}
   \begin{split}
 &\int_{Q_{2r}} y^{1-2s}\nabla w \nabla[((x,y)\cdot \nabla w)\varphi_R]dxdy\\= & \int_{\partial Q_{2r}} y^{1-2s}(\nabla w\cdot \mathbf{n})  [((x,y)\cdot \nabla w)\varphi_R]dS\\
  =&-\lim\limits_{y\rightarrow 0^+}\int_{ B_{2R}\cap(\mathbb{R}^N\times\{y\})} y^{1-2s}\frac{\partial w}{\partial y}   [ ((x,y)\cdot \nabla w)\varphi_R] dx\\
  =&k_{s}^{-1}\int_{ B_{2R}\cap(\mathbb{R}^N\times\{0\})} (x\cdot \nabla_xw)\varphi_R\frac{\partial w}{\partial \nu^{s}} dx\\
  =&k_{s}^{-1}\int_{ B_{2R}\cap(\mathbb{R}^N\times\{0\})} (x\cdot \nabla_xw)\varphi_Rf(x,w) dx\\
  =&k_{s}^{-1}\int_{ B_{2R}\cap(\mathbb{R}^N\times\{0\})} (x\cdot \nabla F(x,u) )\varphi_R  dx-k_{s}^{-1}\int_{ B_{2R}\cap(\mathbb{R}^N\times\{0\})} (x\cdot \nabla_x F(x,u) )\varphi_R  dx\\
  =&-Nk_{s}^{-1}\int_{ B_{2R}\cap(\mathbb{R}^N\times\{0\})}   F(x,u)  \varphi_R  dx-k_{s}^{-1}\int_{ B_{2R}\cap(\mathbb{R}^N\times\{0\})}   F(x,u)  (x\cdot \nabla_x\varphi_R)  dx
  \\&-k_{s}^{-1}\int_{ B_{2R}\cap(\mathbb{R}^N\times\{0\})} (x\cdot \nabla_x F(x,u) )\varphi_R  dx.
 \end{split}
 \end{equation}
For the second integral in the last equality in (\ref{Poh2}), we have
\begin{equation}
   \begin{split}
 \int_{ B_{2R}\cap(\mathbb{R}^N\times\{0\})}   F(x,u) (x\cdot \nabla_x\varphi_R)  dx\leq &C \int_{ (B_{2R}\setminus B_R)\cap(\mathbb{R}^N\times\{0\})}   F(x,u) \frac{|x|}{R}  dx\\
  \leq& C \int_{ (B_{2R}\setminus B_R)\cap(\mathbb{R}^N\times\{0\})}   F(x,u)  dx\rightarrow 0,\text{ as } R\rightarrow+\infty
 \end{split}
 \end{equation}
since $F(x,t)\in L^1(\mathbb{R}^N)$.
As a consequence, from (\ref{Poh2}) we have
\begin{equation}\label{Poh3}
   \begin{split}
 &\lim\limits_{R\rightarrow +\infty}\int_{Q_{2r}} y^{1-2s}\nabla w \nabla[((x,y)\cdot \nabla w)\varphi_R]dxdy\\=&-k_{s}^{-1}\int_{ \mathbb{R}^N}  \left[ NF(x,u)    + (x\cdot \nabla_x F(x,u) )\right] dx.
 \end{split}
 \end{equation}
On the other hand, similar to the proof of Theorem A.1 in \cite{BMU17},   we have
\begin{equation}\label{Poh4}
   \begin{split}
 \lim\limits_{R\rightarrow +\infty}\int_{Q_{2r}} y^{1-2s}\nabla w \nabla[((x,y)\cdot \nabla w)\varphi_R]dxdy =  \frac{2s-N}{2}\int_{\mathbb{R}^{N+1}_+}y^{1-2s}|\nabla w|^2dxdy.
 \end{split}
 \end{equation}
Thanks to (\ref{Poh3}) and (\ref{Poh4}),  we have
\begin{equation}\label{Poh5}
   \begin{split}
  \frac{ N-2s}{2}\int_{\mathbb{R}^{N+1}_+}y^{1-2s}|\nabla w|^2dxdy= k_{s}^{-1}\int_{ \mathbb{R}^N}  \left[ NF(x,u)    + (x\cdot \nabla_x F(x,u) )\right] dx.
 \end{split}
 \end{equation}

\noindent Multiply (\ref{Poh1}) by $w\varphi_R$ and ntegrate by parts to get
\begin{equation}\label{Poh5}
   \begin{split}
 \int_{Q_{2r}} y^{1-2s}\nabla w \nabla(w\varphi_R)dxdy = & \int_{\partial Q_{2r}} y^{1-2s}(\nabla w\cdot \mathbf{n})   w\varphi_R dS\\
   =&k_{s}^{-1}\int_{ B_{2R}\cap(\mathbb{R}^N\times\{0\})} \frac{\partial w}{\partial \nu^{s}} w\varphi_R  dx \\
   =&k_{s}^{-1}\int_{ B_{2R}\cap(\mathbb{R}^N\times\{0\})}f(x,u) u\varphi_R  dx
 \end{split}
 \end{equation}
Proceed now as above to get 
\begin{equation}\label{Poh6}
   \begin{split}
 \int_{\mathbb{R}^{N+1}_+}y^{1-2s}|\nabla w|^2dxdy=  k_{s}^{-1}\int_{ \mathbb{R}^N}  f(x,u)u dx.
 \end{split}
 \end{equation}
Combining (\ref{Poh4}) and (\ref{Poh6}), we deduce that
\begin{equation}\label{Poh6}
   \begin{split}
\frac{ N-2s}{2} \int_{ \mathbb{R}^N}  f(x,u)u dx=\int_{ \mathbb{R}^N}  \left[ NF(x,u)    + (x\cdot \nabla_x F(x,u) )\right] dx.
 \end{split}
 \end{equation}
\end{proof}

\noindent Finally, we prove a crucial local estimate.  This type of estimate has been studied in Proposition 3.1  and Proposition 2.4 in  \cite{TLX11}. Their methods relies on a localization method introduced by Caffarelli
and Silvestre in \cite{CASI07} and the standard Moser iteration. However, these estimates
contain  the extension local domain $Q_R$,
which has no clear interpretation in terms of the original problem in $\mathbb{R}^N$ that is our context. We now give another version of this estimate based on a more direct test function method and Moser's iteration.

\begin{theorem} \label{TH5}
 Assume $a(x)\in L_{loc}^{t}(\mathbb{R}^N)$ for some $t>\frac{N}{2s}$ and that $u\geq 0$ satisfies $$(-\Delta)^s u\leq a(x)u,\ \ x\in \mathbb{R}^N\ .$$
 Then
  \begin{equation}\label{bb1}
   \begin{split}
  \max_{B_r} u(x)\leq C\left(\int_{B_R}|u|^{2_s^*}dx\right)^{\frac{1}{2_s^*}},\ \ 0<r<R,
   \end{split}
 \end{equation}
 where the constant $C>0$ depends only on $N$, $s$, $R$, $t$ and $\|a(x)\|_{L^{t}_{loc}(\mathbb{R}^N)}$.

\end{theorem}

\begin{proof}
For $\beta>1$ and $T>0$, define the function
\begin{equation}
\varphi(t)=\left\{\begin{array}{cccc}
0&\ \ &\text{ if }&  t\leq 0,\\
t^\beta &\ \ &\text{ if }& 0< t\leq T,\\
\beta T^{\beta-1}(t-T)+T^\beta&\ \ &\text{ if }&  t\geq T. \end{array}\right.
\end{equation}
Notice that $\varphi(t)$ is a convex and differentiable function and thus
  \begin{equation} \label{bb2}
   \begin{split}
   (-\Delta)^s\varphi(u)\leq \varphi'(u)(-\Delta)^su.
   \end{split}
 \end{equation}
Let $\eta(x)=\eta(|x|)$ be a smooth cut-off function satisfying $\eta(x)=1$ in $B_r$, $0\leq \eta(x)\leq 1$, $\eta(x)=0$ in $B_R^c$ and $|\eta'|\leq \frac{C}{R-r}$ for some constant $C>0$, where $0<r<R$ has to be determined. For simplicity, in the following, we denote by $\varphi:=\varphi(u(x))$ and $\varphi':=\varphi_u(x)$.

\noindent Choose as test function $\phi(x)=\eta^2\varphi\varphi'$ to obtain
\begin{equation}\label{bb3}
   \begin{split} \int_{\mathbb{R}^N}\eta^2\varphi\varphi'(x)(-\Delta)^sudx\leq \int_{\mathbb{R}^N}\eta^2\varphi\varphi'audx.
     \end{split}
 \end{equation}
However, by (\ref{bb2}),  we have
\begin{equation}\label{bb4}
   \begin{split}  \int_{\mathbb{R}^N}\eta^2\varphi(-\Delta)^s\varphi dx\leq  \int_{\mathbb{R}^N}\eta^2\varphi\varphi'(-\Delta)^sudx.
     \end{split}
 \end{equation}
Using (\ref{bb3}) and (\ref{bb4}), the fact $u\varphi'\leq \beta \varphi$,
by Sobolev embedding theorem and Cauchy inequality, we get
  \begin{equation}\label{bb5}
   \begin{split}S(n,s)\|\eta\varphi\|_{2^*_s}^2\leq &\int_{\mathbb{R}^N}\eta\varphi(-\Delta)^s[\eta\varphi]dx\\
  =&\int_{\mathbb{R}^N}\eta^2\varphi(-\Delta)^s\varphi dx+\frac{1}{2}\int_{\mathbb{R}^N\times\mathbb{R}^N}\frac{ |\eta(x) -\eta(y|^2}{|x-y|^{N+2s}}\varphi(x)\varphi(y)dxdy\\
  \leq&\beta\int_{\mathbb{R}^N} a \eta^2 \varphi^2dx+ \int_{\mathbb{R}^N\times\mathbb{R}^N}\frac{ |\eta(x) -\eta(y|^2}{|x-y|^{N+2s}}\varphi(x)^2dxdy.
    \end{split}
 \end{equation}

\noindent In what follows, we assume $u\in L^{\frac{2\beta t}{t-1}}(\mathbb{R}^N)$ where $\beta$ has to be chosen later on.

\noindent From  $a(x)\in L_{loc}^{t}(\mathbb{R}^N)$ and $\varphi(t)\leq t^\beta$,   we have 
 \begin{equation}\label{bb6}
   \begin{split} \int_{\mathbb{R}^N} a \eta^2 u^{2\beta}dx
  \leq  \left[\int_{\mathbb{R}^N} (\eta a)^{t} dx\right]^{\frac{1}{t}}\left[\int_{\mathbb{R}^N} (\eta u^{2\beta})^{ \frac{t}{t-1}}dx\right]^{\frac{t-1}{t}}.
   \end{split}
 \end{equation}

\noindent Set
 \begin{equation}\label{bb71}
   \begin{split}   &\int_{\mathbb{R}^N }\left( \int_{\mathbb{R}^N }\frac{ |\eta(x) -\eta(y)|^2}{|x-y|^{N+2s}}dy\right)^{t}dx\\&=\int_{|x|\leq R }\left( \int_{\mathbb{R}^N }\frac{ |\eta(x) -\eta(y)|^2}{|x-y|^{N+2s}}dy\right)^{t}dx+\int_{|x|\geq R }\left( \int_{\mathbb{R}^N }\frac{ |\eta(x) -\eta(y)|^2}{|x-y|^{N+2s}}dy\right)^{t}dx\\
   &:=I_1+I_2.
   \end{split}
 \end{equation}
 We obtain 
 \begin{equation}\label{bb72}
   \begin{split}   I_1\leq&\frac{C^t}{(R-r)^{2t}}\int_{|x|\leq R }\left( \int_{ |y-x|\leq R }\frac{ 1}{|x-y|^{N+2s-2}}dy\right)^{t}dx\\&+ C^t \int_{|x|\leq R }\left( \int_{ |y-x|\geq R }\frac{ 1}{|x-y|^{N+2s}}dy\right)^{t}dx\\
   &=\frac{ (2-2s)^{-t} C^t}{(R-r)^{2t}} (N\omega_{N-1})^{t+1} R^{N+(2-2s)t}  +(2s)^{-t}C^t (N\omega_{N-1})^{t+1}R^{N-2st}
   \end{split}
 \end{equation}  and
  \begin{equation}\label{bb7}
   \begin{split}   I_2 =&
   \int_{|x|\geq R }\left( \int_{|y|\leq R}\frac{ |\eta(x) -\eta(y)|^2}{|x-y|^{N+2s}}dy\right)^{t}dx \\
    =&
   \int_{R\leq |x|\leq 2R }\left( \int_{|y|\leq R}\frac{ |\eta(x) -\eta(y)|^2}{|x-y|^{N+2s}}dy\right)^{t}dx
   +\int_{ |x|\geq 2R }\left( \int_{|y|\leq R}\frac{ |\eta(x) -\eta(y)|^2}{|x-y|^{N+2s}}dy\right)^{t}dx\\
   :=&I_3+I_4
   \end{split}
 \end{equation}
The estimate of $I_3$ is similar to the one for $I_1$.  Finally,
   \begin{equation}\label{bb7-}
   \begin{split}   I_4
   \leq&
  \int_{|x|\geq2 R }\left( \int_{|y|\leq R }\frac{ 1}{(|x|-R)^{N+2s}}dy\right)^{t}dx \\
  =&    (N\omega_{N-1})^{t} R^{(N-1)t}\int_{|x|\geq 2R } (|x|-R)^{-(N+2s)t}dx\\
\leq & C R^{N-(1+2s)t}
   \end{split}
 \end{equation}

By combining (\ref{bb4})-(\ref{bb7-}),
  we obtain 
  \begin{equation}\label{bb7}
   \begin{split}   \int_{\mathbb{R}^N }\left( \int_{\mathbb{R}^N }\frac{ |\eta(x) -\eta(y)|^2}{|x-y|^{N+2s}}dy\right)^{t}dx\leq\frac{C}{(R-r)^{2t}}.
   \end{split}
 \end{equation}

\noindent Hence 
  \begin{equation}\label{bb8}
   \begin{split}   \int_{\mathbb{R}^N\times\mathbb{R}^N}\frac{ |\eta(x) -\eta(y|^2}{|x-y|^{N+2s}}\varphi(x)^2dxdy\leq \left[\left( \int_{\mathbb{R}^N }\frac{ |\eta(x) -\eta(y)|^2}{|x-y|^{N+2s}}dy\right)^{t}dx\right]^{\frac{1}{t}}\left(\int_{\mathbb{R}^N}\eta^2 u^{2\beta \frac{t}{t-1}}dx\right)^{\frac{t-1}{t}}.   \end{split}
 \end{equation}

 \noindent Set $$C:=\left(\int_{\mathbb{R}^N} (\eta a(x))^{t} dx\right)^{\frac{1}{t}}+ \left[\int_{\mathbb{R}^N }\left( \int_{\mathbb{R}^N }\frac{ |\eta(x) -\eta(y)|^2}{|x-y|^{N+2s}}dy\right)^{t}dx\right]^{\frac{1}{t}}.$$ Combining (\ref{bb5}), (\ref{bb6}), (\ref{bb7}) and (\ref{bb8}), we get
   \begin{equation}\label{bb9}
   \begin{split}S(n,s)\|\eta\varphi\|_{2^*_s}^2\leq \frac{C}{(R-r)^{2}}\beta\left(\int_{\mathbb{R}^N}(\eta u^{2\beta})^{ \frac{t}{t-1}}dx\right)^{\frac{t-1}{t}}.
    \end{split}
 \end{equation}
Now let $T\rightarrow +\infty$, to obtain 
    \begin{equation}\label{bb10}
   \begin{split} \left(\int_{B_r}u^{\beta 2_s^*}dx\right)^{\frac{1}{\beta2_s^*}}\leq \left[\frac{C\beta}{(R-r)^{2}}\right]^{\frac{1}{2\beta}}\left(\int_{B_R}  u^{2\beta \frac{t}{t-1}}dx\right)^{\frac{t-1}{2t\beta}}
    \end{split}\ .
 \end{equation}
 Since $t>\frac{N}{2s}$,  we set $\beta_i=(\frac{2_s^*(t-1)}{2t})^i$, $i=1,2,\cdots$, $r_i=r_0+\frac{1}{2^i}$. By iterating, we get
   \begin{equation}\label{bb11}
   \begin{split} \left(\int_{B_{r_{m}}}u^{\beta_m 2_s^*}dx\right)^{\frac{1}{\beta_m2_s^*}}\leq \left[ \frac{2_s^*(t-1)C}{2t}\right]^{\frac{1}{2 }\sum_{i=1}^mi/\beta_i}\left(\int_{B_R}  u^{2_s^*}dx\right)^{\frac{1}{2_s^*}}
    \end{split}
 \end{equation}
 Let $m\rightarrow \infty$ to have 
  \begin{equation}\label{bb12}
   \begin{split}
  \max_{B_{r_0}} u(x)\leq C\left(\int_{B_R}u^{2_s^*}dx\right)^{\frac{1}{2_s^*}}.
   \end{split}
 \end{equation}
\end{proof}

 \section{Asymptotic behavior of ground states} \label{S2}

\noindent  Let  $w_\epsilon$ be a positive  minimizer for  $S_{2_s^*-\epsilon}^V$  obtained  in Theorem \ref{ThA}.
Then, by the Lagrange multiplier rule, there exists $\lambda_\epsilon>0$ such that  $w_\epsilon$ is a solution to the equation
\begin{equation}\label{maineq01}(-\Delta)^s u+V(x)u= \lambda_\epsilon u^{2_s^*-1-\epsilon}  \ \ \text{in}\ \ \mathbb{R}^N.\end{equation}
By multiplying both sides of   equation (\ref{maineq01}) by $w_\epsilon$ and then integrating, we get $\lambda_\epsilon=S_{2_s^*-\epsilon}^V$.

\noindent The energy associated with  equation (\ref{maineq01}) is given by 
\begin{equation}\label{funct1}J_\epsilon(u)=\frac{1}{2} \|u\|_{s,V}^2-\frac{1}{2_s^*-\epsilon}S_{2_s^*-\epsilon}^V\|u\|_{2_s^*-\epsilon}^{2_s^*-\epsilon}.\end{equation}
Thus, on the one hand we have
\begin{equation}\label{funct2}J_\epsilon(w_\epsilon)=\frac{1}{2} \|w_\epsilon\|_{s,V}^2-\frac{1}{2_s^*-\epsilon}S_{2_s^*-\epsilon}^V\|w_\epsilon\|_{2_s^*-\epsilon}^{2_s^*-\epsilon}=\frac{2_s^*-\epsilon-2}{2(2_s^*-\epsilon)}S_{2_s^*-\epsilon}^V\ .
\end{equation}

\noindent On the other hand, if $v$ is a nontrivial solution of  (\ref{maineq01}), then it   satisfies $\|v\|_{s,V}^2=S_{2_s^*-\epsilon}^V\|v\|_{2_s^*-\epsilon}^{2_s^*-\epsilon}$ and thus
 \begin{equation}\label{funct3}
   \begin{split} J_\epsilon(v)=&\frac{1}{2} \|v\|_{s,V}^2-\frac{1}{2_s^*-\epsilon}S_{2_s^*-\epsilon}^V\|v\|_{2_s^*-\epsilon}^{2_s^*-\epsilon}\\
   =&\frac{2_s^*-\epsilon-2}{2(2_s^*-\epsilon)}S_{2_s^*-\epsilon}^V\|v\|_{2_s^*-\epsilon}^{2_s^*-\epsilon}
   . \end{split}
 \end{equation}
Besides, we have
 \begin{equation}\label{fun2}
 S_{2_s^*-\epsilon}^V \leq I_\epsilon(v)=\frac {\|v\|_{s,V}^2}{ \|v\|_{2_s^*-\epsilon}^2}=\frac {S_{2_s^*-\epsilon}^V\|v\|_{2_s^*-\epsilon}^{2_s^*-\epsilon}}{ \|v\|_{2_s^*-\epsilon}^2}=S_{2_s^*-\epsilon}^V\|v\|_{2_s^*-\epsilon}^{2_s^*-\epsilon-2},
\end{equation}
which yields that $\|v\|_{2_s^*-\epsilon}\geq 1 $. Thus,   we have $J_\epsilon(v)\geq \frac{2_s^*-\epsilon-2}{2(2_s^*-\epsilon)}S_{2_s^*-\epsilon}^V$ by (\ref{funct3}). This fact together with (\ref{funct2}) implies that $w_\epsilon$ is a ground state of   equation (\ref{maineq01}). Furthermore, if we set 
$$u_\epsilon=\left(S_{2_s^*-\epsilon}^V\right)^{-\frac{1}{1-2_s^*+\epsilon}}w_\epsilon$$ 
then, $u_\epsilon$ is a ground state of   equation (\ref{maineq0}). Observe that $I_\epsilon(u_\epsilon)= S_{2_s^*-\epsilon}^V$.

\noindent For each fixed $\epsilon \in (0,2_s^*-2)$, by means of the
mountain-pass theorem,     (\ref{maineq0}) admits a positive  ground state  (see e.g.~Theorem 1.4 in \cite{PEAJ12}). However, we don’t know whether the mountain-pass  solution and  the minimal solution $u_\epsilon $ obtained above do agree since uniqueness is not known. Anyway, in what follows, we will focus on the minimal solution $u_\epsilon$.
We remark that in the special case $V(x)=1$,   the ground state is unique and  radially symmetric, see  \cite{FRLE}.

\begin{lemma}\label{lem1} For  any fixed  $\epsilon \in (0,2_s^*-2)$,
any nontrivial  $u_\epsilon$  of   (\ref{maineq0}) satisfies
\begin{equation}\label{supbound}\|u_\epsilon\|_\infty\geq V_0^{\frac{1}{2_s^*-2 }}.\end{equation}
\end{lemma}

\begin{proof} Since $u_\epsilon$ enjoys (\ref{maineq0}), we have
$$\|u_\epsilon\|^2_{s,V_0}\leq \|u_\epsilon\|_{s,V}^2= \|u_\epsilon\|_{2_s^*-\epsilon}^{2_s^*-\epsilon}$$
which yields 
$V_0\|u_\epsilon\|_2^2\leq \|u_\epsilon\|_{2_s^*-\epsilon}^{2_s^*-\epsilon},$
that is,
$$\int_{\mathbb{R}^N}u_\epsilon^2(V_0-u_\epsilon^{2_s^*-2-\epsilon})dx\leq 0.$$
Thus, we get $\|u_\epsilon\|_\infty\geq V_0^{\frac{1}{2_s^*-2-\epsilon}}\geq V_0^{\frac{1}{2_s^*-2 }},$ and the result follows.
\end{proof}

\noindent We next need the following result proved in \cite{COTA04}.
\begin{lemma} \label{COTA04}\cite{COTA04}
  The infimum in (\ref{bestconst}) is attained, that is
  $$S= \frac {[\widetilde{u}]_{s}^2}{ \|\widetilde{u}\|_{2^*_s}^2},$$ where
  $$\widetilde{u}(x):=\kappa(\mu^2+|x-x_0|^2)^{\frac{2s-N}{2}}, \ \ x\in \mathbb{R}^N$$ with $\kappa\in \mathbb{R}\setminus \{0\}$, $\mu>0$ and $x_0 \in \mathbb{R}^N$ fixed constant.
  Equivalently, the function $\overline{u}$ defined by
  $\overline{u}(x):=\frac{\widetilde{u}}{ \|\widetilde{u}\|_{2^*_s}}$ is such that
  \begin{equation}\label{bestconst1}S=\inf_{ u\in D^{s}(\mathbb{R}^N)}\{[u]_{s}^2: \|u\|_{2^*_s}=1\}.\end{equation}Furthermore,
 the function
  $$u^*(x):=\overline{u}\left(S^{-\frac{1}{2s}}x\right),\ \ x\in \mathbb{R}^N$$ is a solution of
$$(-\Delta u)^s u=|u|^{2_s^*-2}u, \ \ \text{in}\ \ \mathbb{R}^N$$ satisfying the property
$$\|u\|_{2^*_s}^{2_s^*}=S^{\frac{N}{2s}}.$$

\end{lemma}

\begin{proposition}\label{lm6}
 $\lim\limits_{\epsilon \rightarrow 0^+} S_{2_s^*-\epsilon}^V=S$.
\end{proposition}

\begin{proof}
Choose $\phi\in C_0^\infty(\mathbb{R}^N)$, $\phi\geq 0$ such that $\inf\limits_{0<\epsilon<2^*_s-2}\|\phi\|_{2^*_s-\epsilon}>0$,   to get
\begin{equation}\label{bound1}0<S_{2_s^*-\epsilon}^V\leq \frac {\|\phi\|_{s,V}^2}{ \inf\limits_{0<\epsilon<2^*_s-2}\|\phi\|_{2^*_s-\epsilon}^2}<+\infty.\end{equation}
This means that $\{S_{2_s^*-\epsilon}^V\}$ is uniformly bounded with respect to $\epsilon$. Next, we further prove that $\lim\limits_{\epsilon \rightarrow 0^+}  S_{2_s^*-\epsilon}^V=S$.

\noindent Let $w\in H_{V}^s(\mathbb{R}^n)$ be such that $\|w\|_{2_s^*-\epsilon}=1$ and $\|w\|_{s,V}^2=S_{2_s^*-\epsilon}^V$. Then,
\begin{equation}
\|w\|_2^2\leq \|w\|_s^2 \leq \max\left\{1,\frac{1}{V_0}\right\}\|w\|_{s,V}^2=:CS_{2_s^*-\epsilon}^V.
\end{equation}
 By H\"{o}lder's inequality we have
\begin{equation} \label{inf1+}
1=\|w\|_{2_s^*-\epsilon}^{2_s^*-\epsilon}\leq  \|w\|_{2}^{\frac{2\epsilon}{ 2_s^*-2 }}\|w\|_{2_s^*}^{\frac{2_s^*(2_s^*-2-\epsilon)}{2_s^*-2 }}\leq (CS_{2_s^*-\epsilon}^V)^{\frac{\epsilon}{2_s^*-2}}\|w\|_{2_s^*}^{\frac{2_s^*(2_s^*-2-\epsilon)}{2_s^*-2 }}.
\end{equation}
Thanks to (\ref{bound1}) and (\ref{inf1+}), we get
\begin{equation} \label{inf1}
 1\leq \liminf_{\epsilon \rightarrow 0^+} \|w\|_{2_s^*} .
 \end{equation}
On the other hand, by (\ref{bestconst}), we have
\begin{equation}\label{inf3}
S\leq \frac{[w]_{s}^2}{\|w\|_{2^*_s}^2}\leq \frac{\|w\|_{s,V}^2}{\|w\|_{2^*_s}^2}=\frac{S_{2_s^*-\epsilon}^V}{\|w\|_{2^*_s}^2},
 \end{equation}
Thanks to (\ref{inf1}) and (\ref{inf3}), we have
\begin{equation}\label{bound2}
S\leq \liminf_{\epsilon \rightarrow 0^+}S_{2_s^*-\epsilon}^V\ .
 \end{equation}

\noindent Next we prove that
\begin{equation} \label{bound3}
 \limsup_{\epsilon \rightarrow 0^+}S_{2_s^*-\epsilon}^V\leq S.
 \end{equation}
Once (\ref{bound3}) is proved, the result follows from (\ref{bound2}).

\noindent Let
$$U_\epsilon(x)=\epsilon^{\frac{2s-N}{2}} u^*(x/\epsilon),$$ where $u^*$ is defined in Lemma  \ref{COTA04}.
Furthermore, let
$\eta(x)\in C_0^\infty(\mathbb{R}^N)$ be such that $0\leq \eta(x) \leq 1$ in $\mathbb{R}^N$, $\eta(x)\equiv 1$ in $B_{1/2}$ and $\eta(x)\equiv 0$ in $B_{1}^c$. Set $u_\epsilon(x):=\eta(x)U_\epsilon(x),$ $ x\in \mathbb{R}^N $.
Then, as $\epsilon\rightarrow 0^+$ we have
\begin{equation}\label{estim1}
[u_\epsilon]_{s}^2\leq S^{\frac{N}{2s}}+O(\epsilon^{N-2s}),
 \end{equation}
\begin{equation}\label{estim2}
 \int_{\mathbb{R}^N }V(x)|u_\epsilon(x)|^2dx=
C_s \epsilon^{2s}+O(\epsilon^{N-2s}), \ \ \text{if}\ \ N>4s,
\end{equation}
and
\begin{equation}\label{estim3}
 \|u_\epsilon(x)\|_{2_s^*}^{2_s^*} =S^{\frac{N}{2s}}+O(\epsilon^N),
 \end{equation}
for some positive constant $C_s$ depending only on $s,$ see Propositions 21 and 22 in \cite{SV15} or Lemma 2.4 in \cite{doms14}.
By Taylor's expansion we get
\begin{equation}\label{estim4}
 \|u_\epsilon(x)\|_{2_s^*-\epsilon}^2=  \|u_\epsilon(x)\|_{2_s^*}^2+O(\epsilon).
 \end{equation}
Hence, we deduce from (\ref{estim1})--(\ref{estim4}) that
\begin{equation}
 \limsup_{\epsilon \rightarrow 0^+} S_{2_s^*-\epsilon}^V\leq  \limsup_{\epsilon \rightarrow 0^+}\frac{\|u_\epsilon\|_{s,V}^2}{\|u_\epsilon\|^2_{2_s^*-\epsilon}}\leq S,
 \end{equation}
which implies  (\ref{bound3}).

\end{proof}

\noindent Recalling that  $u_\epsilon$  is a solution to  (\ref{maineq0}) and that $u_\epsilon$ attains $S_{2^*_s-\epsilon}^V$,  we get
\begin{equation} \label{bound0}\|u_\epsilon\|^2_{s,V}=\|u_\epsilon\|_{2^*_s-\varepsilon}^{2^*_s-\epsilon} \ \ \text{ and } \ \ \|u_\epsilon\|^2_{s,V}=S_{2^*_s-\epsilon}^V\|u_\epsilon\|_{2^*_s-\varepsilon}^{2}.  \end{equation}
So, we have
\begin{equation} \label{bound4}\|u_\epsilon\|_{s,V}=\left(S_{2^*_s-\varepsilon}^V\right)^{\frac{2_s^*-\epsilon}{2(2_s^*-2-\epsilon)}} \ \ \text{ and } \ \  \|u_\epsilon\|_{2^*_s-\varepsilon}=\left(S_{2^*_s-\varepsilon}^V\right)^{\frac{1}{2_s^*-2-\epsilon}}. \end{equation}
 These facts together with Lemma \ref{lm6} imply the following
 
\begin{lemma}\label{SS2}
\begin{equation} \label{normbound}
 \lim_{\epsilon \rightarrow 0^+}\|u_\epsilon\|_{s,V}= S^{\frac{N}{4s}},\ \ \ \    \lim_{\epsilon \rightarrow 0^+}\|u_\epsilon\|_{2^*_s-\epsilon}= S^{\frac{N-2s}{4s}}.
 \end{equation}
\end{lemma}

\noindent Now let us prove that $\|u_\epsilon\|_\infty$ blows up as $\epsilon\rightarrow 0^+$, namely 

\begin{lemma}\label{unbound}
$ \lim_{\epsilon \rightarrow 0^+}\|u_\epsilon\|_\infty=+\infty$.
\end{lemma}

\begin{proof}
Suppose by contradiction the claim does not hold true. Then, there exists a sequence $\epsilon_j\rightarrow 0^+$ such that $\|u_{\epsilon_j}\|_\infty$ stays bounded. Let $x_{\epsilon_j}$ be a maximum point of $u_{\epsilon_j}$.  Define $w_{\epsilon_j}(x)=u_{\epsilon_j}(x+x_{\epsilon_j})$, then $ \|w_{\epsilon_j}\|_\infty$ is bounded as well and
 \begin{equation}\label{bound0}(-\Delta)^s w_{\epsilon_j}=-V(x+x_{\epsilon_j})w_{\epsilon_j}+w_{\epsilon_j}^{2_s^*-1-\epsilon_j} \  \ \text{in} \ \ \mathbb{R}^N. \end{equation}
Now,  since $V\in C^2(\mathbb{R}^N)\cap L^\infty(\mathbb{R}^N)$, we have that $\|(-\Delta)^s w_{\epsilon_j}\|_{\infty}$ is uniformly bounded with respect to $\epsilon_j$. As a consequence of this fact and of standard regularity results (see e.g.~Lemma 4.4 in \cite{CASI14}), we deduce that $\|w_{\epsilon_j}\|_{C^{2,\alpha}}$ is uniformly bounded with respect to $\epsilon_j$, for
some $\alpha \in (0,1)$.

\noindent By (\ref{normbound}), $[w_{\epsilon_j}]_s=[u_{\epsilon_j}]_s$ and $\|w_{\epsilon_j}\|_2=\|u_{\epsilon_j}\|_2$ are bounded. Thus, $\{w_{\epsilon_j}\}$ is   bounded in $H_{V}^s(\mathbb{R}^N)$. Up to extracting a subsequence, which we still denote by $\{w_{\epsilon_j}\}$,  one has $w_{\epsilon_j}\rightharpoonup w_0$ in $H_{V}^s(\mathbb{R}^N)$, $w_{\epsilon_j}\rightarrow  w_0$  a.e.~in $\mathbb{R}^N$ and
 $w_{\epsilon_j}\rightarrow w_0$ in $C^{2,\alpha}_{loc}(\mathbb{R}^N)$.
Moreover, by   (\ref{supbound}) one has $w_0(0)\geq V_0^{\frac{1}{2_s^*-2 }}>0$.

\noindent Let us now distinguish two cases:

\textit{Case 1.} $\{x_{\epsilon_j}\}_j$ is bounded. Up to a subsequence, we may assume that $x_{\epsilon_j}\rightarrow x_0$. Then, $w_0$ is a nonnegative classical solution of
 \begin{equation}\label{bound01}(-\Delta)^s w_0=-V(x+x_0)w_0+w_0^{2_s^*-1} \ \  \text{ in }\ \ \mathbb{R}^N. \end{equation}
It follows from the maximum principle that $w_0>0$. Thus, by Lemma \ref{SS2} we have
 \begin{equation}\label{bound02}
S\leq \frac{[w_0]_{s}^2}{\|w_0\|_{2^*_s}^2}<\|w_0\|_{2^*_s}^{2_s^*-2}\leq \liminf_{j\rightarrow \infty}\|w_{\epsilon_j}\|_{2^*_s}^{2_s^*-2} =S,
 \end{equation} which is a contradiction.

\textit{Case 2.} $\{x_{\epsilon_j}\}_j$ is unbounded. Up to a subsequence, we may assume that $x_{\epsilon_j}\rightarrow \infty$.
 Then, by (\ref{bound0}) and the dominated convergence theorem, we have
  \begin{equation}\label{bound03}   \begin{split}
 [w_0]_{s}^2\leq &-V_0  \|w_0\|_{2}^2+ \|w_0\|_{2^*_s}^{2_s^*} +\lim_{j\rightarrow \infty}\int_{\mathbb{R}^N}(V(x)-V(x+x_{\epsilon_j})) w_{\epsilon_j} w_0dx\\
\leq &-V_0  \|w_0\|_{2}^2+ \|w_0\|_{2^*_s}^{2_s^*},  \end{split}
 \end{equation} which yields $[w_0]_{s}^2< \|w_0\|_{2^*_s}^{2_s^*}$ and similarly to the proof of (\ref{bound02}), we get a contradiction.

\end{proof}

\noindent As $\epsilon\rightarrow \epsilon_0\in (0,\frac{2_s^*-2}{2})$, we have that $\{u_\epsilon\}$ is uniformly bounded with respect to $\epsilon$, as established in the following 

\begin{lemma}\label{unbound00} There exists $K>0$, which does not depend on $\epsilon$, such that any solutions $u_\epsilon$ of (\ref{maineq0}) satisfies
$\|u_\epsilon\|_\infty\leq K$ as $\epsilon\rightarrow \epsilon_0$.
\end{lemma}

\begin{proof} The claim can be achieved via Moser's iteration.
Indeed, let   $w_\epsilon$ be the harmonic extension of $u_\epsilon$, see e.g.~\cite{CASI07}. Then, $w_\epsilon$ satisfies
\begin{equation}\label{Poh1}
\left\{\begin{array}{cccc}
-\text{div}(y^{1-2s}\nabla w_\epsilon)=0&\ \ &\text{ in }&   \mathbb{R}^{N+1}_+,\\
\frac{\partial w_\epsilon}{\partial\nu^{s}}=   -V(\cdot)w_\epsilon(\cdot,0)+w_\epsilon^{2_s^*-1-\epsilon}(\cdot,0)&\ \ &\text{ in }&   \mathbb{R}^N\times \{y=0\}, \end{array}\right.
\end{equation}
where
$$\frac{\partial w_\epsilon}{\partial\nu^{s}}:=-\frac{1}{k_{s}}\lim_{y\rightarrow 0^+}y^{1-2s}\frac{\partial w_\epsilon}{\partial y}(x,y).$$
and $k_{s}=\frac{2^{1-2s}\Gamma(1-s)}{\Gamma(s)}$.

\noindent Following Corollary 2.1 in \cite{ALMI16}, for each $L>0$, we set
\begin{equation}
w_{\epsilon,L}(x,y)=\left\{\begin{array}{cccc}
w_\epsilon(x,y)&\ \ &\text{ if }&  w_\epsilon(x,y)\leq L,\\
L&\ \ &\text{ if }&  w_\epsilon(x,y)\geq L,  \end{array}\right.\ \ u_{\epsilon,L}(x)=w_{\epsilon,L}(x,0),
\end{equation}
and $\psi_{\epsilon,L}=w_{\epsilon,L}^{2(\beta-1)} w_\epsilon $, where $\beta>1$ to be determined later on.
By testing with $\psi_{\epsilon,L}$, we get
 \begin{equation}
   \begin{split}
  \int_{ \mathbb{R}^{N+1}_+} y^{1-2s}\nabla  w_\epsilon \nabla \psi_{\epsilon,L} dxdy= k_s^{-1}\int_{ \mathbb{R}^{N}} [-V(x) u_\epsilon(x)+u_\epsilon^{2_s^*-1-\epsilon}(x)]\psi_{\epsilon,L}(x,0) dx.
 \end{split}
 \end{equation}
Thus,
 \begin{equation}\label{Mors1}
   \begin{split}
  \int_{ \mathbb{R}^{N+1}_+} y^{1-2s}\nabla  w_\epsilon \nabla (w_{\epsilon,L}^{2(\beta-1)} w_\epsilon) dxdy\leq   k_s^{-1}\int_{ \mathbb{R}^{N}}u_\epsilon^{2_s^*-\epsilon} u_{\epsilon,L}^{2(\beta-1)}  dx.
 \end{split}
 \end{equation}
Note that
 \begin{equation}
 \nabla  w_\epsilon \nabla (w_{\epsilon,L}^{2(\beta-1)} w_\epsilon)=\left\{\begin{array}{cccc}
(2\beta-1)w_{\epsilon,L}^{2(\beta-1)}(x,y)|\nabla w_{\epsilon}|^2&\ \ &\text{ if }&  w_\epsilon(x,y)\leq L,\\
L^{2\beta-1}|\nabla w_{\epsilon}|^2&\ \ &\text{ if }&  w_\epsilon(x,y)\geq L. \end{array}\right.
\end{equation}
Thus,
from (\ref{Mors1}),   Sobolev imbedding (see e.g. ~(2.9) in \cite{WSK14}) and H\"{o}lder's inequality,
 \begin{equation}\label{Mors2}
   \begin{split}
\left(\int_{ \mathbb{R}^{N}}| u_{\epsilon,L}^{\beta-1} u_\epsilon |^{2_s^*} dx\right)^{\frac{2}{2_s^*}} \leq& C(N,s)
  \int_{ \mathbb{R}^{N+1}_+} y^{1-2s}|\nabla (w_{\epsilon,L}^{\beta-1} w_\epsilon)|^2 dxdy \\ \leq&  \beta C(N,s) \int_{ \mathbb{R}^{N}}u_\epsilon^{2_s^*-2-\epsilon} u_\epsilon^2 u_{\epsilon,L}^{2(\beta-1)} dx\\
  \leq&  \beta C(N,s)\left( \int_{ \mathbb{R}^{N}}u_\epsilon^{2_s^*} dx\right)^{\frac{2^*_s-2-\epsilon}{2^*_s}}\left( \int_{ \mathbb{R}^{N}}( u_{\epsilon,L}^{\beta-1}u_\epsilon )^{\frac{ 22^*_s}{2+\epsilon}}dx\right)^{\frac{2+\epsilon}{2^*_s}}.
 \end{split}
 \end{equation}
Since $\|u_\epsilon\|_{2_s^*}$ is bounded, from  (\ref{Mors2}) we get
 \begin{equation}\label{Mors3}
   \begin{split}
\| u_{\epsilon,L}^{\beta-1} u_\epsilon\|_{2_s^*}^2    \leq   \beta C(N,s) \left( \int_{ \mathbb{R}^{N}}( u_{\epsilon,L}^{\beta-1}u_\epsilon )^{\frac{ 22^*_s}{2+\epsilon}}dx\right)^{\frac{2+\epsilon}{2^*_s}}.
 \end{split}
 \end{equation}
 As $u_{\epsilon}\in L^{\frac{22^*_s\beta}{\epsilon}}(\mathbb{R}^N)$, by using the fact that $w_{\epsilon,L}\leq w_{\epsilon}$, we get
 \begin{equation}\label{Mors4}
   \begin{split}
\| u_{\epsilon,L}^{\beta-1} u_\epsilon\|_{2_s^*}^2    \leq   \beta C(N,s) \left( \int_{ \mathbb{R}^{N}}u_{\epsilon}^{\frac{22^*_s\beta}{2+\epsilon}}dx\right)^{\frac{2+\epsilon}{2^*_s}}.
 \end{split}
 \end{equation}
Let $L\rightarrow +\infty$ and apply Fatou's lemma to get 
 \begin{equation}\label{Mors5}
   \begin{split}
\|  u_\epsilon\|_{2_s^*\beta}^2    \leq   \beta^{\frac{1}{\beta}} C^{\frac{1}{\beta}}(N,s) \|  u_\epsilon\|_{\frac{22^*_s}{2+\epsilon}\beta}^2.
 \end{split}
 \end{equation}
The claim now follows by iteration: let $\beta_i=(\frac{2+\epsilon}{2})^i$, $i=1,2,\cdots$, then
 \begin{equation}\label{Mors5}
   \begin{split}
\|  u_\epsilon\|_{2_s^*\beta_{m+1}}    \leq   \left(\frac{2+\epsilon}{2}\right)^{\frac{1}{2}\sum_{i=1}^m i(\frac{2+\epsilon}{2})^{-i}} C^{ \frac{1}{2}\sum_{i=1}^m(\frac{2+\epsilon}{2})^{-i}}(N,s) \|  u_\epsilon\|_{ 2^*_s }.
 \end{split}
 \end{equation}
Passing to the limit as $m\rightarrow +\infty$ in (\ref{Mors5}), we have 
 \begin{equation*}\label{Mors6}
   \begin{split}
\|  u_\epsilon\|_{\infty}    \leq  C \|  u_\epsilon\|_{ 2^*_s }.
 \end{split}
 \end{equation*}
which concludes the proof.
\end{proof}

\begin{lemma} \label{SS0}
Let $\epsilon_0>0$, then $\limsup\limits_{\epsilon\rightarrow\epsilon_0} S_{2_s^*-\epsilon}^V\leq S_{2_s^*-\epsilon_0}^V$.
\end{lemma}

\begin{proof}
Let $\phi>0$ be such that $S_{2_s^*-\epsilon_0}^V=\frac {\|\phi\|_{s,V}^2}{ \|\phi\|_{2_s^*-\epsilon_0}^2}$.  Then,
 \begin{equation}
   \begin{split}
 \int_{\mathbb{R}^N}|\phi|^{2^*-\epsilon}dx= \int_{\mathbb{R}^N}|\phi|^{2^*-\epsilon_0} dx +(\epsilon-\epsilon_0)\int_{\mathbb{R}^N}|\phi|^{2^*-\epsilon_0+t(\epsilon_0-\epsilon)}\ln \phi\, dx,
 \end{split}
 \end{equation} where $t\in (0,1)$.
Since $|\phi^{\alpha}\ln \phi|\leq C$ for any $\alpha>0$ as $\phi\rightarrow 0^+$ and $\ln \phi \leq 1+\phi$ as $\phi\geq1$, recalling that $\|\phi\|_\infty$ is bounded by Lemma \ref{unbound00}, we get 
 \begin{equation} \label{SS}
   \begin{split}
 \int_{\mathbb{R}^N}|\phi|^{2^*-\epsilon}dx= \int_{\mathbb{R}^N}|\phi|^{2^*-\epsilon_0} dx +O(\epsilon-\epsilon_0).
 \end{split}
 \end{equation}
Now, by (\ref{SS}) and the definition of $S_{2^*_s-\epsilon}^V$, we have
 \begin{equation}
   \begin{split}
 \limsup\limits_{\epsilon\rightarrow\epsilon_0}\left(S_{2^*_s-\epsilon}^V\right)^{\frac{2_s^*-\epsilon_0}{2}}\leq & \limsup\limits_{\epsilon\rightarrow\epsilon_0} \frac {\|\phi\|_{s,V}^{2_s^*-\epsilon_0}}{ \|\phi\|_{2_s^*-\epsilon}^{2_s^*-\epsilon_0}}\\
 =&  \limsup\limits_{\epsilon\rightarrow\epsilon_0} \frac {\|\phi\|_{s,V}^{2_s^*-\epsilon_0}}{ [\|\phi\|_{2_s^*-\epsilon_0}+O(\epsilon-\epsilon_0)]^{2_s^*-\epsilon_0}}= \left(S_{2_s^*-\epsilon_0}^V\right)^{\frac{2_s^*-\epsilon_0}{2}}.
 \end{split}
 \end{equation}
The proof of Lemma \ref{SS0} is complete.

\end{proof}

\noindent Let $x_\epsilon$ be the global maximum point of $u_\epsilon$ and let $\mu_\epsilon>0$ be such that
$$u_\epsilon(x_\epsilon)=\|u_\epsilon\|_{\infty}=\mu_\epsilon^{-\frac{2s}{2_s^*-2-\epsilon}}.$$ Clearly, from Lemma \ref{unbound} $\mu_\epsilon\rightarrow 0$ as $\epsilon\rightarrow 0^+.$
Set
$$v_\epsilon(x)=\mu_\epsilon^{\frac{2s}{2_s^*-2-\epsilon}}u_\epsilon(x_\epsilon+\mu_\epsilon x).$$
Then $0<v_\epsilon(x)\leq 1$, $v_\epsilon(0)=1$ and $v_\epsilon$ satisfies the following 
\begin{equation}\label{res1}
   \begin{split}
 (-\Delta)^s v_\epsilon+\mu_\epsilon^{2s}V(x_\epsilon+\mu_\epsilon x)v_\epsilon=v_\epsilon^{2_s^*-1-\epsilon} \ \ \text{in} \ \ \mathbb{R}^N.
 \end{split}
 \end{equation}
We have that $\|(-\Delta)^s v_{\epsilon}\|_{\infty}$ is  uniformly bounded with respect to  $\epsilon$.  As a consequence of this fact and regularity results, we deduce that also $\|v_{\epsilon}\|_{C^{2,\alpha}}$ is uniformly bounded with respect to  $\epsilon$, for
some $\alpha \in (0,1)$.  Similarly to the proof of Lemma \ref{unbound}, there exists a sequence $\epsilon$, still denoted by $v_\epsilon$, such that $v_{\epsilon}\rightarrow U$ in $C_{loc}^{2,\alpha}(\mathbb{R}^N)$, where $U$ is the
positive solution of equation
\begin{equation}
  (-\Delta)^su= u^{2^*_s-1} \ \  \text{in}\ \ \mathbb{R}^N
 \end{equation}
and $U(0)=\|U\|_\infty=1$.
\noindent From Theorem 1.2 in \cite{CLO06}, 
\begin{equation}\label{att}U(x)=\left(1+\frac{|x|^2}{\lambda^2}\right)^{\frac{2s-N}{2}},\ \ \text{ where }\ \ \lambda=2\left(\frac{\Gamma\left(\frac{N+2s}{2}\right)}{\Gamma\left(\frac{N-2s}{2}\right)}\right)^{\frac{1}{2}}. \end{equation}

\noindent Since $$S=\frac{[U]_{s}^2}{\|U\|_{2_s^*}^{2}}=\|U\|_{2_s^*}^{2_s^*-2}=[U]_{s}^{2-\frac{4}{2_s^*}},$$  we conclude that
$$\|U\|_{2_s^*}^{2_s^*}=[U]_{s}^{2}=S^{\frac{N}{2s}}.$$
 By Lemma  \ref{SS2},  we have
\begin{equation} \label{cover1}
   \begin{split}
  S^{\frac{N}{2s}}=&[U ]_{s}^2\leq \liminf\limits_{\epsilon\rightarrow 0^+}[v_\epsilon]_{s}^2\\
  \leq &  \limsup\limits_{\epsilon\rightarrow 0^+}[v_\epsilon]_{s}^2\\
   \leq &  \limsup\limits_{\epsilon\rightarrow 0^+}\left[\int_{\mathbb{R}^N}\frac{|v_\epsilon(x)-v_\epsilon(y)|^2}{|x-y|^{N+2s}}dxdy+\int_{\mathbb{R}^N}V(x_\epsilon+\mu_\epsilon x)v^2_\epsilon dx\right]\\
   =&  \limsup\limits_{\epsilon\rightarrow 0^+}\mu_\epsilon^{\frac{(N-2s)\epsilon}{2_s^*-2-\epsilon} }\left[\int_{\mathbb{R}^N}\frac{|u_\epsilon(x)-u_\epsilon(y)|^2}{|x-y|^{N+2s}}dxdy+ \int_{\mathbb{R}^N}V( x)u^2_\epsilon dx\right]\\
  \leq &  \limsup\limits_{\epsilon\rightarrow 0^+} \left[\int_{\mathbb{R}^N}\frac{|u_\epsilon(x)-u_\epsilon(y)|^2}{|x-y|^{N+2s}}dxdy+ \int_{\mathbb{R}^N}V( x)u^2_\epsilon dx\right]\\
  =&S^{\frac{N}{2s}}.
 \end{split}
 \end{equation}

\noindent Finally,  form Lemma \ref{SS2} and (\ref{cover1}), we obtain the following convergences

\begin{lemma}\label{conver2}
$[v_\epsilon-U]_{s} \rightarrow 0$,
$\|v_\epsilon-U\|_{2_s^*} \rightarrow 0$,
$[v_\epsilon]_{s}^2 \rightarrow S^{\frac{N}{2s}}$, and $\mu_\epsilon^\epsilon\rightarrow 1$
 as $\epsilon\rightarrow 0^+.$
\end{lemma}

\begin{proposition}\label{conver2+} If $x_{\epsilon}\rightarrow x_0\in (0,+\infty)$ as $\epsilon\rightarrow 0^+$. Then
$$|u_\epsilon|^{2_s^*}(x)\rightarrow S^{\frac{N}{2s}} \delta(x-x_0), \quad \epsilon\rightarrow 0$$ in the sense of distributions.
\end{proposition}
\begin{proof} For any $\phi\in C_0^\infty(\mathbb{R}^N)$, we get
 \begin{equation} \label{cover4-}
   \begin{split}
\lim_{\epsilon\rightarrow 0^+}\int_{\mathbb{R}^N} |u_\epsilon|^{2_s^*} \phi dx=& \lim_{\epsilon\rightarrow 0^+}\left[ \mu_\epsilon^{-\frac{\epsilon N(N-2s)}{4s-\epsilon(N-2s)}}\int_{\mathbb{R}^N} |v_\epsilon|^{2_s^*} \phi(x_\epsilon+\mu_\epsilon x) dx\right]\\
=&  \phi(x_0)\int_{\mathbb{R}^N} |U|^{2_s^*}  dx\\
=& \phi(x_0) S^{\frac{N}{2s}}.
  \end{split}
 \end{equation}
\end{proof}

\noindent Notice that up to now we do not know wether the global maximum point    $x_\epsilon$ turns out to be  bounded or unbounded.

\begin{lemma}\label{conver3} Suppose that $\{x_\epsilon\}$ is bounded. Then,
 \begin{equation} \label{cover4}
   \begin{split}
  \sup_{\epsilon \in \left(0,\frac{2_s^*-2}{2}\right)}\int_{|x|\geq R} u_\epsilon^{2_s^*}dx\rightarrow 0   \text{ as } R\rightarrow +\infty.
 \end{split}
 \end{equation}
\end{lemma}
\begin{proof}
Assume by contradiction  that \eqref{cover4} does not hold. Then, there exist two sequences $\epsilon_j\rightarrow\epsilon_0$ and $R_j\rightarrow +
\infty$ such that
 \begin{equation} \label{cover55}
   \begin{split}
  \int_{|x|\geq R_j} u_{\epsilon_j}^{2_s^*}dx \geq \delta,
 \end{split}
 \end{equation}
for some $\delta>0$ and $j=1,2,\cdots$. We distinguish two cases:

\noindent \textit{Case 1}. $\varepsilon_0>0$. By (\ref{normbound}), $\{u_\epsilon\}$ is  bounded in $H^s_{V}(\mathbb{R}^N)$, passing to a subsequence $\{u_{\epsilon_j}\}$ if necessary, we may assume 
$u_{\epsilon_j}\rightharpoonup u_{\epsilon_0}$ in $H^s_{V}(\mathbb{R}^N)$. On the other hand, from Lemma \ref{unbound00}, we know that $\|u_{\varepsilon_j}\|_\infty$ is  bounded and by regularity we deduce that $\|u_{\epsilon_j}\|_{C^{2,\alpha}}$ is uniformly bounded with respect to $\epsilon_j$, for
some $\alpha \in (0,1)$. Up to extracting again a subsequence, still denoted by $\{u_{\epsilon_j}\}$,   we have $u_{\epsilon_j}\rightarrow u_{\epsilon_0}$ in $C_{loc}^{2,\alpha}(\mathbb{R}^N)$. Thus, $u_{\epsilon_0}$ is a classical nonnegative solution of the equation
 \begin{equation}\label{decay1}(-\Delta)^s u+V(x)u=u^{2_s^*-1-\epsilon_0} \ \ \text{in}\ \ \mathbb{R}^N. \end{equation}
By
(\ref{supbound}), we get $u_{\epsilon_0}(x)\not\equiv 0$. Moreover, if there is $x_0\in \mathbb{R}^N$ such that $u_{\epsilon_0}(x_0)= 0$, then from (\ref{decay1}), $(-\Delta)^s u(x_0)=0$. However, by the very definition
$$(-\Delta)^s u(x_0)=c_{N,s}\text{PV}\int_{\mathbb{R}^n}\frac{ -u(y)}{|x_0-y|^{N+2s}}dy<0, $$
which is a contradiction. Thus, $u_{\epsilon_0}(x)> 0$ for all $x\in \mathbb{R}^N$.

\noindent Now, by (\ref{bound4}) and Lemma \ref{SS0},  observe that
 \begin{equation} \label{decay2}
   \begin{split}
 S_{2^*_s-\epsilon_0}^V\leq& \frac {\|u_{\epsilon_0}\|_{s,V}^2}{ \|u_{\epsilon_0}\|_{2_s^*-\epsilon_0}^2}=\|u_{\epsilon_0}\|_{s,V}^{2-\frac{4}{2_s^*-\epsilon_0}}\\
 \leq& \liminf\limits_{j\rightarrow \infty}\|u_{\epsilon_j}\|_{s,V}^{2-\frac{4}{2_s^*-\epsilon_0}}\\
=& \liminf\limits_{j\rightarrow \infty}S_{2_s^*-\epsilon_j}^V\\
\leq& \limsup\limits_{\epsilon \rightarrow \epsilon_0}S_{2_s^*-\epsilon}^V\\
\leq& S_{2_s^*-\epsilon_0}^V.
 \end{split}
 \end{equation}
Therefore, we get
 \begin{equation} \label{decay3}
   \begin{split}
 \lim\limits_{j\rightarrow \infty}S_{2_s^*-\epsilon_j}^V= S_{2^*_s-\epsilon_0}^V.
 \end{split}
 \end{equation}
Similarly to the proof of Lemma \ref{SS0}, from (\ref{decay3}) we get $\|u_{\epsilon_j}\|_{s,V}\rightarrow \|u_{\epsilon_0}\|_{s,V}$ as $j\rightarrow+\infty$ and hence $u_{\epsilon_j}\rightarrow u_{\epsilon_0}$ in $L^{2_s^*}(\mathbb{R}^N)$ as $j\rightarrow+\infty$. This contradicts (\ref{cover55}).

\noindent \textit{Case 2}. $\varepsilon_0=0$. Thanks to Lemma  \ref{conver2}, we obtain a contradiction  from (\ref{cover55}). Indeed, we have
 \begin{equation} \label{cover5}
   \begin{split}
  \delta \leq& \int_{|x|\geq R_j} u_{\epsilon_j}^{2_s^*}dx
   =\mu_{\epsilon_j}^{N-\frac{2s2_s^*}{2_s^*-2-\epsilon_j}}\int_{|x_{\epsilon_j}+\mu_{\epsilon_j}x|\geq R_j} v_{\epsilon_j}^{2_s^*}(x)dx\\
   \leq &(\mu_{\epsilon_j}^{\epsilon_j})^ {\frac{N(N-2)}{4s-N\epsilon_j+2s\epsilon_j}}\int_{|x|\geq \frac{R_j-|x_{\epsilon_j}|}{\mu_{\epsilon_j}}} v_{\epsilon_j}^{2_s^*}(x)dx\\
   \rightarrow &0 \text{ as } j\rightarrow\infty,
 \end{split}
 \end{equation}
since $\mu_{\epsilon_j}^{\epsilon_j}\rightarrow 1$, $\frac{R_j-|x_{\epsilon_j}|}{\mu_{\epsilon_j}}\rightarrow +\infty$ and $v_{\epsilon_j}\rightarrow U$ in $L^{2_s^*}(\mathbb{R}^N)$ as $j\rightarrow \infty$.

\noindent The proof   is now complete.

\end{proof}

\begin{lemma}\label{conver3+} Suppose that $\{x_\epsilon\}$ is unbounded. Then, for any fixed $R>0$,
 \begin{equation} \label{cover4+}
   \begin{split}
  \limsup_{\epsilon\rightarrow 0^+}\int_{|x-x_\epsilon|\geq R} u_\epsilon^{2_s^*}dx= 0.
 \end{split}
 \end{equation}
\end{lemma}

\begin{proof}The proof is similar to Lemma \ref{conver3}. Suppose that the claim is not true. Then there exist a sequence $\epsilon_j\rightarrow\epsilon_0$ such that
 \begin{equation} \label{cover5}
   \begin{split}
  \int_{|x-x_{\epsilon_j}|\geq R} u_{\epsilon_j}^{2_s^*}dx \geq \delta,
 \end{split}
 \end{equation}
for some $\delta>0$ and $j=1,2,\cdots$.

  The proof of the case $\epsilon_0>0$ is similar  to Lemma \ref{conver3}. For  $\epsilon_0=0$, we have
 \begin{equation} \label{cover5+}
   \begin{split}
  \delta \leq& \int_{|x-x_{\epsilon_j}|\geq R} u_{\epsilon_j}^{2_s^*}dx
   =\mu_{\epsilon_j}^{N-\frac{2s2_s^*}{2_s^*-2-\epsilon_j}}\int_{|x|\geq \frac{R}{\mu_{\epsilon_j}}} v_{\epsilon_j}^{2_s^*}(x)dx\\
   = &(\mu_{\epsilon_j}^{\epsilon_j})^ {\frac{N(N-2)}{4s-N\epsilon_j+2s\epsilon_j}}\int_{|x|\geq \frac{R}{\mu_{\epsilon_j}}} v_{\epsilon_j}^{2_s^*}(x)dx\\
   \rightarrow &0, \text{ as } j\rightarrow\infty,
 \end{split}
 \end{equation}
since $\mu_{\epsilon_j}^{\epsilon_j}\rightarrow 1$, $\frac{R}{\mu_{\epsilon_j}}\rightarrow +\infty$ and $v_{\epsilon_j}\rightarrow U$ in $L^{2_s^*}(\mathbb{R}^N)$ as $j\rightarrow \infty$.

\end{proof}

\noindent The following lemmas will play an important role in our analysis.
\begin{lemma} \label{DEC0} Assume that $\{x_\epsilon\}$ stays bounded. Then, there exist    constants $C, R>0$ independent of $\epsilon$, such that
 \begin{equation}\label{decay0}
   \begin{split}
  |u_\epsilon(x)|\leq \frac{C}{|x|^{N+2s}},\ \ \text{for }\ |x|\geq R.
   \end{split}
 \end{equation}

\end{lemma}

\begin{proof}

 We observe that
$$(-\Delta)^s u_\epsilon\leq u_{\epsilon}^{2_s^*-2-\epsilon}u_\epsilon\ .$$
Since
  $u_{\epsilon}^{2_s^*-2-\epsilon}\in L_{loc}^t(\mathbb{R}^N)$ for some $t>\frac{N}{2s}$,  from Theorem \ref{TH5} we have
 \begin{equation}\label{decay0-0}
   \begin{split}
  \max_{B_r(y)} u_\epsilon(x)\leq C\left(\int_{B_R(y)}|u_\epsilon|^{2_s^*}dx\right)^{\frac{1}{2_s^*}},\ \ \forall y\in \mathbb{R}^N, 0<r<R,
   \end{split}
 \end{equation} where $C$ is independent of $\epsilon$.
Thus,
we conclude from
(\ref{cover4}) and (\ref{decay0-0}) that
 \begin{equation}\label{decay01}
   \begin{split}
  \sup_{\epsilon \in \left(0,\frac{2_s^*-2}{2}\right)} u_\epsilon(y)\rightarrow 0, \ \text{ as } |y|\rightarrow +\infty.
   \end{split}
 \end{equation}
This fact together with Lemma C.2 in \cite{FRLE} imply 
 \begin{equation}\label{decay02}|u_\epsilon(x)|\leq \frac{C}{|x|^{N+2s}}. \end{equation}
Actually, we   first fix $\epsilon>0$ to applying Lemma C.2 in \cite{FRLE}, and then we take the supremum with respect to $\epsilon$. Finally,  from  Lemmas \ref{SS0} and \ref{SS2},  we get
(\ref{decay02}). See also \cite{MFEV15}.

\end{proof}

\begin{lemma} \label{DEC} Suppose that $\{x_\epsilon\}$ is unbounded. Then there exists   a constant $C>0$ independent of $\epsilon$ such that for small $\epsilon>0$,
 \begin{equation}\label{decay0+}
   \begin{split}
  |u_\epsilon(x)|\leq \frac{C}{| x-x_\epsilon|^{N+2s}},\ \ \text{for }\ |x-x_\epsilon|\geq R,
   \end{split}
 \end{equation}
for any $R>0$.

\end{lemma}

\begin{proof}

From (\ref{cover4+}), for any $\delta, R>0$, there exists a small $\epsilon_0>0$ such that if $0<\epsilon<\epsilon_0$ and $|y-x_\epsilon|\geq R$, then
 \begin{equation}\label{decay06}
   \begin{split}
  u_\epsilon(y)\leq \delta.
   \end{split}
 \end{equation}

\noindent  Let $w_\epsilon(x)=u_\epsilon(y)$, $y=V_0^{-\frac{1}{2s}} x$. Then $w_\epsilon(x)$ enjoys the following 
  \begin{equation}\label{decay07}
   \begin{split}
  (-\Delta)^s w_\epsilon+V(y)V_0^{-1}w_\epsilon= V_0^{-1} w_\epsilon^{2_s^*-1-\epsilon}.
   \end{split}
 \end{equation}

\noindent Furthermore,
by condition $(V_1)$,  if we choose $\delta>0$ sufficiently small and $R_1>0$ large enough, we have 
  \begin{equation}\label{decay07-1}(-\Delta)^s w_\epsilon+w_\epsilon=f_\epsilon(x):=\left[1-V(y)V_0^{-1}\right]w_\epsilon(x)+V_0^{-1}w_\epsilon^{2_s^*-1-\epsilon}(x)\leq 0 \end{equation}
for small $\epsilon>0$, $|x|\geq R_1$ and $|V_0^{-\frac{1}{2s}} x-x_\epsilon|\geq R$.

\noindent Borrowing some results from \cite{PEAJ12}, we also have 
 \begin{equation}\label{decay07+1}w_\epsilon(x)=\mathcal{K}\ast f_\epsilon(x)=\int_{\mathbb{R}^N}\mathcal{K}(x-y)f_\epsilon(y)dy, \end{equation}
where $\mathcal{K}$ is the Bessel kernel and which enjoys the following properties:
\begin{itemize}
\item[$(K_1)$] $\mathcal{K}$ is positive, radially symmetric and smooth in $\mathbb{R}^N\setminus \{0\}$;
\item[$(K_2)$] There is $C_1, C_2>0$ such that
 \begin{equation}\label{k1}\mathcal{K}(x)\leq \frac{C_1}{|x|^{N+2s}},\ \ \text{if }\ |x|\geq 1 \end{equation} and
 \begin{equation}\label{k2}\mathcal{K}(x)\leq \frac{C_2}{|x|^{N-2s}},\ \ \text{if }\ |x|\leq 1. \end{equation}
\end{itemize}

\noindent From (\ref{decay07-1}) and (\ref{decay07+1}) we have
 \begin{equation}\label{decay08}
   \begin{split}
 w_\epsilon(x)\leq \int_{\{|V_0^{-\frac{1}{2s}} y-x_\epsilon|\leq R, |y|\geq R_1\}}\mathcal{K}(x-y)f_\epsilon(y)dy+\int_{ \{ |y|\leq R_1\}}\mathcal{K}(x-y)f_\epsilon(y)dy
   \end{split}
 \end{equation}
Note that $|V_0^{-\frac{1}{2s}}y-x_\epsilon|>R \Leftrightarrow |y-V_0^{\frac{1}{2s}}x_\epsilon|>V_0^{\frac{1}{2s}}R $. Since $\{x_\epsilon\}$ is unbounded, then there exists $0<\epsilon_1<\epsilon_0$ such  that $|x_\epsilon|\geq R+R_1$ for $0<\epsilon<\epsilon_1$. Thus, for $|y|\leq R_1$, we get  $|y-x_\epsilon|\geq |x_\epsilon|-|y|\geq R$.
So, from (\ref{decay07}) and  (\ref{decay07-1}),  we obtain
 \begin{equation}\label{decay09}
   \begin{split}
 \int_{ \{ |y|\leq R_1\}}\mathcal{K}(x-y)f_\epsilon(y)dy\leq C\int_{ \{ |y|\leq R_1\}}\mathcal{K}(x-y)dy.
   \end{split}
 \end{equation}

\noindent By (\ref{k1}) and (\ref{k2}),   we have
  \begin{equation}\label{decay09}
   \begin{split}
 \int_{ \{ |y|\leq R_1\}}\mathcal{K}(x-y)dy= &\int_{ \{ |y|\leq R_1, |x-y|<1\}}\mathcal{K}(x-y)dy+\int_{ \{ |y|\leq R_1, |x-y|\geq1\}}\mathcal{K}(x-y)dy\\
  \leq& \frac{N\omega_{N-1}C_1}{2s}+C_2 N\omega_{N-1}R_1^N.
   \end{split}
 \end{equation}

\noindent Besides,  for  $R_2>R$,   $|x-V_0^{\frac{1}{2s}}x_\epsilon|>V_0^{\frac{1}{2s}}R_2$ and  $|y-V_0^{\frac{1}{2s}}x_\epsilon|\leq V_0^{\frac{1}{2s}}R$, we get $|x-y|\geq \frac{R_2-R}{R}|x-V_0^{\frac{1}{2s}}x_\epsilon|$ and thus
\begin{equation}\label{decay10}
   \begin{split}
&\int_{\{|V_0^{-\frac{1}{2s}} y-x_\epsilon|\leq R, |y|\geq R_1\}}\mathcal{K}(x-y)f_\epsilon(y)dy\\& \leq C\int_{\{|V_0^{-\frac{1}{2s}} y-x_\epsilon|\leq R, |y|\geq R_1\}}\mathcal{K}(x-y)u_\epsilon^{2_s^*-1-\epsilon}(y)dy\\
&\leq C \|u_\epsilon\|_{2_s^*-\epsilon}^{2_s^*-1-\epsilon}\left(\int_{\{|V_0^{-\frac{1}{2s}} y-x_\epsilon|\leq R, |y|\geq R_1\}}\mathcal{K}(x-y)^{\frac{2N-(N-2s)\epsilon}{N-2s}} dy\right)^{\frac{N-2s}{2N-(N-2s)\epsilon}}\\
&\leq \frac{C}{|V_0^{-\frac{1}{2s}}x-x_\epsilon|^{N+2s}}.
   \end{split}
 \end{equation}
Combining (\ref{decay08})-- (\ref{decay10}), for    $|V_0^{-\frac{1}{2s}}x-x_\epsilon|>R_2$ and small $\epsilon>0$, we have
$$|w_\epsilon(x)|\leq \frac{C}{|V_0^{-\frac{1}{2s}}x-x_\epsilon|^{N+2s}}.$$
That is,
$$|u_\epsilon(x)|\leq  \frac{C}{|x-x_\epsilon|^{N+2s}},\ \ \ |x-x_\epsilon|>R_2.$$

\noindent Since $R$ is arbitrary, as well as $R_2$ is arbitrary, the proof is complete.

\end{proof}

\begin{remark}
By using standard comparison arguments as in \cite{PEAJ12}, we can prove results similar to  (\ref{cover4}) and  (\ref{cover4+}). However, the constant $C$ obtained there may depend on $\epsilon$.
\end{remark}

\begin{lemma}\label{comp} These exists a positive constant $C$ independent of $\epsilon$, such that
 \begin{equation}\label{compare}v_\epsilon(x)\leq CU(x),\ \ x\in \mathbb{R}^N. \end{equation}
\end{lemma}
\begin{proof} Note that we do not assume that $\{x_\epsilon\}$ is bounded or unbounded.    From the definition of $v_\epsilon$ and $U$, $v_\epsilon(0)=U(0)=1$, and since $v_\epsilon(x)\in C^{2,\alpha}$, by choosing some large $C$,  (\ref{compare}) holds in a neighborhood of zero. Therefore, it is enough to establish (\ref{compare}) ifor $|x|$ bounded away from zero. For this purpose, let $\Phi_\epsilon(x)$ be the Kelvin transform of $v_\epsilon$, namely 
\begin{equation}\label{compare0}\Phi_\epsilon(x)=|x|^{2s-N}v_\epsilon\left(\frac{x}{|x|^2}\right). \end{equation}
Then, $\Phi_\epsilon$ satisfies
\begin{equation}\label{compare1}
   \begin{split}
 (-\Delta)^s \Phi_\epsilon+\mu_\epsilon^{2s}|x|^{-4s}V\left(x_\epsilon+\mu_\epsilon \frac{x}{|x|^2}\right)\Phi_\epsilon=|x|^{(2s-N)\epsilon}\Phi_\epsilon^{2_s^*-1-\epsilon} \ \ \text{in}\ \ \mathbb{R}^N.
 \end{split}
 \end{equation}
Now, we aim at proving that $\{\Phi_\epsilon\}$ is uniformly bounded with respect to  $\epsilon$ in a neighborhood of zero, and this will imply (\ref{compare}) by (\ref{compare0}).

\noindent From (\ref{compare1}), we obtain
\begin{equation}\label{compare2}
   \begin{split}
 (-\Delta)^s \Phi_\epsilon\leq |x|^{(2s-N)\epsilon}\Phi_\epsilon^{2_s^*-1-\epsilon}:=a(x)\Phi_\epsilon, \ \ a(x)=|x|^{(2s-N)\epsilon}\Phi_\epsilon^{2_s^*-2-\epsilon}.
 \end{split}
 \end{equation}

\noindent {\it Claim:} $a(x)\in  L_{loc}^t(\mathbb{R}^N)$ with some $t>\frac{N}{2s}$.

\noindent Assume for the moment the claim holds true and let us complete the proof.  By Theorem \ref{TH5},
  for any compact set $K$, we have
 \begin{equation}\label{decay00}
   \begin{split}
  \max_{K} \Phi_\epsilon(x)&\leq C\left(\int_{K}|\Phi_\epsilon|^{2_s^*}dx\right)^{\frac{1}{2_s^*}}\leq C\left(\int_{\mathbb{R}^N}|v_\epsilon|^{2_s^*}dx\right)^{\frac{1}{2_s^*}}
  \\&\leq C(\mu_\epsilon^\epsilon)^{\frac{ N}{2_s^*(2_s^*-2-\epsilon)}}\left(\int_{\mathbb{R}^N}|u_\epsilon|^{2_s^*}dx\right)^{\frac{1}{2_s^*}}\leq C.
   \end{split}
 \end{equation}
The last inequality follows from the facts $\mu_\epsilon^\epsilon\rightarrow 1$  and $\|u_\epsilon\|_{2_s^*}\leq C\|u_\epsilon\|_{s,V(x)}\rightarrow CS^{\frac{N}{4s}}$ as $\epsilon\rightarrow 0^+.$

\noindent Thus, it remains to prove the claim. On the one hand, for $r>0$ we get
 \begin{equation}\label{compare3}
   \begin{split}
 \int_{\mu_\epsilon^2\leq |x|\leq r} a(x)^tdx\leq& (\mu_\epsilon^\epsilon)^{2(2s-N)t} \int_{\mu_\epsilon^2\leq |x|\leq r}  \Phi_\epsilon^{(2_s^*-2-\epsilon)t}dx\\
 \leq& (\mu_\epsilon^\epsilon)^{2(2s-N)t}|B_r|^{1-\frac{(2_s^*-2-\epsilon)t}{2_s^*}} \left(\int_{B_r}  \Phi_\epsilon^{2_s^*}dx\right)^{\frac{(2_s^*-2-\epsilon)t}{2_s^*}}
  \leq C,
 \end{split}
 \end{equation}
since  $\mu_\epsilon^\epsilon\rightarrow 1$ as $\epsilon\rightarrow 0^+$ and $\Phi_\epsilon\rightarrow \overline{U}$ in $L^{2_s^*}(\mathbb{R}^N)$, where $\overline{U}(x)=|x|^{2s-N}U\left(\frac{x}{|x|^2}\right)$.

On the other hand, if $\{x_\epsilon\}$ is bounded and $|x|\leq \frac{\mu_\epsilon}{R-|x_\epsilon|}$, or  if $\{x_\epsilon\}$ is unbounded and $|x|\leq \frac{\mu_\epsilon}{R}$, by Lemmas \ref{DEC0} and \ref{DEC},  we have
 \begin{equation}\label{compare3}
   \begin{split}
  \Phi_\epsilon(x)&=|x|^{2s-N}v_\epsilon\left(\frac{x}{|x|^2}\right)\\
  &=\mu_\epsilon^{\frac{2s}{2_s^*-2-\epsilon}}|x|^{2s-N}u_\epsilon\left(x_\epsilon+\mu_\epsilon\frac{x}{|x|^2}\right)\\
  &\leq  C\mu_\epsilon^{\left[\frac{2s}{2_s^*-2-\epsilon}-(N+2s)\right]}|x|^{4s}.
 \end{split}
 \end{equation}
Thus,   we have
 \begin{equation}\label{compare3}
   \begin{split}
 \int_{|x|\leq \mu_\epsilon^2} a(x)^tdx\leq& C \mu_\epsilon^{\left[\frac{2s}{2_s^*-2-\epsilon}-(N+2s)\right](2_s^*-2-\epsilon)t} \int_{|x|\leq \mu_\epsilon^2}
 |x|^{[4s(2_s^*-2-\epsilon)+(2s-N)\epsilon]t}
 dx\leq C
 \end{split}
 \end{equation}
and the proof is complete.

\end{proof}

\begin{proposition} \label{asy1}
Assume $N>4s$ and suppose that $x_{\epsilon}\rightarrow x_0$ as $\epsilon\rightarrow 0^+$. Then,  $$\lim_{\epsilon\rightarrow 0^+} \epsilon \|u_{\epsilon}\|_{\infty}^{\frac{4s}{N-2}}=A_{N,s}\left[V(x_0)+\frac{1}{2s}x_0\cdot \nabla V(x_0)\right],$$
where $$A_{N,s}= \frac{2^{2(N+1)}N^2\pi^{\frac{N}{2}} \Gamma\left(\frac{N-4s}{2}\right)}{(N-2s)^2\Gamma(N-2s)} S^{-\frac{N}{2s}}  .$$

\end{proposition}
\begin{proof}
By Pohozaev identity (\ref{Poh}), we have
 \begin{equation}\label{asy2}
   \begin{split}
  &\left(\frac{1}{2_s^*-\epsilon}-\frac{1}{2_s^*}\right)\int_{ \mathbb{R}^N}   u_\epsilon^{2_s^*-\epsilon} dx\\&= \int_{ \mathbb{R}^N} \left[V(x)+\frac{1}{2s}x\cdot \nabla V(x)\right] u_\epsilon^2 dx\\
  &=\mu_\epsilon^N\int_{ \mathbb{R}^N} \left[V(x_\epsilon+\mu_\epsilon x)+\frac{1}{2s}(x_\epsilon+\mu_\epsilon x)\cdot \nabla V(x_\epsilon+\mu_\epsilon x)\right] u_\epsilon^2(x_\epsilon+\mu_\epsilon x) dx\\
   &=\mu_\epsilon^{N-\frac{4s}{2_s^*-2-\epsilon}}\int_{ \mathbb{R}^N} \left[[V(x_\epsilon+\mu_\epsilon x)+\frac{1}{2s}(x_\epsilon+\mu_\epsilon x)\cdot \nabla V(x_\epsilon+\mu_\epsilon x)\right] v_\epsilon^2 dx.
 \end{split}
 \end{equation}
Since $N>4s$, by Lemma \ref{comp} and the Lebesgue dominated convergence theorem, we get
  \begin{equation}\label{asy3}
   \begin{split}
    &\lim_{\epsilon\rightarrow 0^+}\int_{ \mathbb{R}^N} \left[V(x_\epsilon+\mu_\epsilon x)+\frac{1}{2s}(x_\epsilon+\mu_\epsilon x)\cdot \nabla V(x_\epsilon+\mu_\epsilon x)\right] v_\epsilon^2 dx\\
     &=  \left[V(x_0)+\frac{1}{2s}x_0\cdot \nabla V(x_0)\right] \int_{ \mathbb{R}^N}U^2 dx.
 \end{split}
 \end{equation}
By direct calculations, we deduce that
  \begin{equation}\label{asy4}
   \begin{split}
      \int_{ \mathbb{R}^N}U^2 dx=& \lambda^{2N} \int_{ \mathbb{R}^N}(1+|x|^2)^{2s-N}dx\\
     =& \lambda^{2N}\omega_N\int_{0}^\infty (1+r^2)^{2s-N} r^{N-1}dr\\
      =& \frac{1}{2}\lambda^{2N} \omega_N\int_{0}^\infty (1+s)^{2s-N} s^{-1+\frac{N}{2}}ds\\
      =&  \lambda^{2N}  \frac{\pi^{\frac{N}{2}}}{\Gamma(\frac{N}{2})}  B\left(\frac{N}{2}, \frac{N}{2}-2s\right)\\
      =&2^{2N}\pi^{\frac{N}{2}} \left[\frac{\Gamma\left(\frac{N+2s}{2}\right)}{\Gamma\left(\frac{N-2s}{2}\right)}\right]^{N}\frac{\Gamma\left(\frac{N-4s}{2}\right)}{\Gamma(N-2s)}.
 \end{split}
 \end{equation}

\noindent Finally, combine (\ref{asy2})--(\ref{asy4}) to have as $\epsilon \to 0$
 \begin{equation}\label{asy6}
   \begin{split}
  \epsilon \mu_\epsilon^{2s}=&\left(\frac{2N}{N-2s}\right)^2 \left[V(x_0)+\frac{1}{2s}x_0\cdot \nabla V(x_0)\right] S^{-\frac{N}{2s}}\int_{ \mathbb{R}^N}U^2 dx+o(1) \\
  =&  \frac{2^{2(N+1)}N^2\pi^{\frac{N}{2}} \Gamma\left(\frac{N-4s}{2}\right)}{(N-2s)^2\Gamma(N-2s)} \left[\frac{\Gamma\left(\frac{N+2s}{2}\right)}{\Gamma\left(\frac{N-2s}{2}\right)}\right]^{N} \left[V(x_0)+\frac{1}{2s}x_0\cdot \nabla V(x_0)\right] S^{-\frac{N}{2s}} \\ &+o(1).
 \end{split}
 \end{equation}
 This concludes the proof of Lemma  \ref{asy1}.

\end{proof}

\begin{remark} \label{REM1}From the proof of Proposition
\ref{asy1},
assuming $N>4s$, no matter $x_\epsilon$ stays bounded or not, we still have 
$$\epsilon=O(\mu_\epsilon^{2s}).$$

\end{remark}

\noindent {\it Proof of Theorem \ref{th2}.} The conclusions (1) and (2)  in Theorem \ref{th2} follow from  Propositions \ref{lm6}  and \ref{asy1}. Clearly, Corollary \ref{th1} is a particular case of Theorem \ref{th2}.

\section{Localizing blow up points} \label{S3}

\noindent We next recall for convenience of the reader a few basic facts on fractional Sobolev spaces. Let $\beta>0$ and $p\in[1,\infty)$, 
$$\mathcal{W}^{\beta,p}(\mathbb{R}^N):=\{u\in L^p(\mathbb{R}^N):  \mathscr{F}^{-1}[(1+|\xi|^{\beta}) \hat{u}]\in  L^p(\mathbb{R}^N)\}$$
endowed with the norm
$$\|u\|_{\mathcal{W}^{\beta,p}(\mathbb{R}^N)}=\|\mathscr{F}^{-1}[(1+|\xi|^{\beta}) \hat{u}]\|_p \ .$$

\noindent We refer to \cite{PEAJ12} for the following results.
\begin{proposition} \label{lm21}
The following properties hold true:
 \begin{itemize} \item[$(1)$]
 If $0<\beta<1$, $1<p\leq q\leq \frac{Np}{N-\beta p}<\infty$ or $p=1$ and $1\leq q<\frac{N}{N-\beta}$, then $\mathcal{W}^{\beta,p}(\mathbb{R}^N)$ is continuously embedded in $L^q(\mathbb{R}^N)$.

 \item[$(2)$] Assume that $0\leq \beta\leq 2$ and $\beta>\frac{N}{p}$. If $\beta-\frac{N}{p}>1$  and  $0<\mu\leq\beta-1-\frac{N}{p}$,  then  $\mathcal{W}^{\beta,p}(\mathbb{R}^N)$ is continuously embedded in $C^{1,\mu}(\mathbb{R}^N)$.  If $\beta-\frac{N}{p}<1$ and  $0<\mu\leq\beta-\frac{N}{p}$, then  $\mathcal{W}^{\beta,p}(\mathbb{R}^N)$ is continuously embedded in $C^{0,\mu}(\mathbb{R}^N)$.
 \end{itemize}
\end{proposition}

\noindent For $p \in [1,+\infty)$ and $\beta>0$, consider the Bessel potential space
$$\mathcal{L}^{\beta,p}(\mathbb{R}^N)=\{u\in L^p(\mathbb{R}^N):  \mathscr{F}^{-1}[(1+|\xi|^{2})^{\frac{\beta}{2}} \hat{u}]\in  L^p(\mathbb{R}^N)\}.$$ Then,
$\mathcal{L}^{\beta,p}(\mathbb{R}^N)=\mathcal{W}^{\beta,p}(\mathbb{R}^N)$, see Theorem 3.1  in \cite{PEAJ12}. On the other hand, from    Theorem 5 in Chapter V  of \cite{Stein70}, for   $p\in[2,\infty)$ and $0<\beta<1$, one has $\mathcal{W}^{\beta,p}(\mathbb{R}^N)\subset W^{\beta,p}(\mathbb{R}^N)$, where $W^{\beta,p}(\mathbb{R}^N)$ is the usual fractional Sobolev space defined by
$$W^{\beta,p}(\mathbb{R}^N)=\left\{u\in L^p(\mathbb{R}^N):  \int_{\mathbb{R}^N\times\mathbb{R}^N}\frac{|u(x)-u(y)|^p}{|x-y|^{n+\beta p}}dxdy\right\}.$$

\noindent Our next target is to identify the location of the blow up points. For this purpose we adapt the method developed in  \cite{NITA93}, where
the basic idea is to get an asymptotic expansion of the ground state energy and then to compare it with an upper bound of $S_{2_s^*-\epsilon}^V$. This method has been used to deal with the localization of blow-up points of ground states to semilinear problems in \cite{WANG96}.

\noindent Let us begin with establishing an upper bound for $S_{2_s^*-\epsilon}^V$.
\begin{theorem} \label{th5}Assume $N>4s$ and that $u_{\epsilon_j}$ is a ground state of (\ref{maineq0}) satisfying (\ref{fun00}) which has a maximum point $x_{\epsilon_j}$ which enjoys $x_{\epsilon_j}\rightarrow x_0$ as $j \rightarrow \infty$. Then
    \begin{equation}\label{Lo3}
   \begin{split}
   &S_{2^*_s-\epsilon_j}^V\leq  S\\&+\mu_j^{2s}\left\{  S^{\frac{2s-N}{2s}}V(\hat{x}_0) \int_{\mathbb{R}^N}  U^2dy - \widetilde{C}_{N,s} S \left[ \frac{2}{(2_s^*)^2}\ln S^{\frac{N}{2s}}-\frac{2}{2_s^*}S^{-\frac{N}{2s}}\int_{\mathbb{R}^N}   U^{2_s^*}\ln U  dx\right]\right\}\\&+o(\mu_j^{2s}),
   \end{split}
 \end{equation}
where $\hat{x}_0$ is a global minimum point of $V(x)$ and  $$\widetilde{C}_{N,s}=\Psi_{N,s}\left[V(x_0)+\frac{1}{2s}x_0\cdot \nabla V(x_0)\right].$$

\end{theorem}

\begin{proof}
Let
$$\phi_j(x)=U\left(\frac{x-\hat{x}_0}{\mu_j}\right).$$
Then by inspection
  \begin{equation}\label{Lo1}
   \begin{split}
   [\phi_j]_{s}^2=\mu_j^{N-2s}[U]_{s}^2=\mu_j^{N-2s}S^{\frac{N}{2s}},
   \end{split}
 \end{equation}
and by dominated convergence, we have
  \begin{equation}\label{Lo2}
   \begin{split}
   \int_{\mathbb{R}^N} V(x)\phi^2dx&=\mu_j^N \int_{\mathbb{R}^N} V(\hat{x}_0+\mu_jy)U^2dy\\
   &=V(\hat{x}_0) \mu_j^N\int_{\mathbb{R}^N}  U^2dy+o(\mu_j^N).
   \end{split}
 \end{equation}

\noindent By
(\ref{asy6}), we also get 
  \begin{equation}\label{Lo3}
   \begin{split}
   \epsilon_j=\widetilde{C}_{N,s} \mu_j^{2s}+o(\mu_j^{2s}).
   \end{split}
 \end{equation}
 Thus, by using Taylor's formula, we get
   \begin{equation}\label{Lo4}
   \begin{split}
  &\mu_j^{\frac{-2N}{2_s^*-\epsilon_j}}\left( \int_{\mathbb{R}^N} |\phi|^{2_s^*-\epsilon_j}dx\right)^{\frac{2}{2_s^*-\epsilon_j}}\\
  &= \left( \int_{\mathbb{R}^N} U^{2_s^*-\epsilon_j}dx\right)^{\frac{2}{2_s^*-\epsilon_j}}\\
 &= \left[ \int_{\mathbb{R}^N} (U^{2_s^*}-\epsilon_j U^{2_s^*}\ln U) dx+o(\epsilon_j)\right]^{\frac{2}{2_s^*-\epsilon_j}}\\
 &= \left[  \left(S^{\frac{N}{2s}}-\epsilon_j\int_{\mathbb{R}^N}   U^{2_s^*}\ln U  dx\right)^{\frac{2}{2_s^*-\epsilon_j}}+o(\epsilon_j)\right]\\
 &= \left[ S^{\frac{N-2s}{2s}}+  \epsilon_jS^{\frac{N-2s}{2s}} \left( \frac{2}{(2_s^*)^2}\ln S^{\frac{N}{2s}}-\frac{2}{2_s^*}S^{-\frac{N}{2s}}\int_{\mathbb{R}^N}   U^{2_s^*}\ln U  dx\right)+o(\epsilon_j)\right] \\
 &= S^{\frac{N-2s}{2s}} \left[ 1+   \epsilon_j  \left( \frac{2}{(2_s^*)^2}\ln S^{\frac{N}{2s}}-\frac{2}{2_s^*}S^{-\frac{N}{2s}}\int_{\mathbb{R}^N}   U^{2_s^*}\ln U  dx\right)+o(\epsilon_j)\right].
   \end{split}
 \end{equation}
So, by (\ref{Lo3}), $\mu_j^{\frac{2N}{2_s^*-\epsilon_j}}=\mu_j^{N-2s}+o(\epsilon_j)$,   we have
   \begin{equation}\label{Lo5}
   \begin{split}
  &\left( \int_{\mathbb{R}^N} |\phi|^{2_s^*-\epsilon_j}dx\right)^{\frac{2}{2_s^*-\epsilon_j}}\\
 &= \mu_j^{N-2s}S^{\frac{N-2s}{2s}} \left[ 1+ \widetilde{C}_{N,s} \mu_j^{2s}  \left( \frac{2}{(2_s^*)^2}\ln S^{\frac{N}{2s}}-\frac{2}{2_s^*}S^{-\frac{N}{2s}}\int_{\mathbb{R}^N}   U^{2_s^*}\ln U  dx\right)+o(\mu_j^{2s})\right].
   \end{split}
 \end{equation}

\noindent  By the very definition of $S_{2_s^*-\epsilon_j}^V$, we have

    \begin{multline}\label{Lo6}
   S_{2^*_s-\epsilon_j}^V\leq \frac {\|u_{j}\|_{s,V}^2}{ \|u_{j}\|_{2_s^*-\epsilon_j}^2}\\
 \leq \frac{\mu_j^{N-2s}S^{\frac{N}{2s}}+V(\hat{x}_0) \mu_j^N \int_{\mathbb{R}^N}  U^2dy+o(\mu_j^N)}{\mu_j^{N-2s} S^{\frac{N-2s}{2s}} \left[ 1+   \widetilde{C}_{N,s}\mu_j^{2s} \left( \frac{2}{(2_s^*)^2}\ln S^{\frac{N}{2s}}-\frac{2}{2_s^*}S^{-\frac{N}{2s}}\int_{\mathbb{R}^N}   U^{2_s^*}\ln U  dx\right)+o(\mu_j^{2s})\right]}\\
 =\left[ S +\mu_j^{2s}S^{\frac{2s-N}{2s}}V(\hat{x}_0) \int_{\mathbb{R}^N}  U^2dy+o(\mu_j^{2s})\right]\\
 \cdot \left[ 1-   \widetilde{C}_{N,s}\mu_j^{2s}  \left( \frac{2}{(2_s^*)^2}\ln S^{\frac{N}{2s}}-\frac{2}{2_s^*}S^{-\frac{N}{2s}}\int_{\mathbb{R}^N}   U^{2_s^*}\ln U  dx\right)+o(\mu_j^{2s})\right].
\end{multline}
This concludes the proof. 
\end{proof} 

\noindent For simplicity, set $\mu_j:=\mu_{\epsilon_j}$, $x_j:=x_{\epsilon_j}$.
For  $v_j(x)=\mu_j^{\frac{2s}{2_s^*-2-\epsilon}}u_j(x_j+\mu_j x)$, let $v_j=U+\mu_j^{2s}w_j$, then by (\ref{res1}) we have 
\begin{equation}\label{res1+}
   \begin{split}
 (-\Delta)^s w_j-(2_s^{*}-1) U^{2_s^*-2}w_j+V(x_j+\mu_j x)v_j=F(w_j) \ \ \text{in}\ \ \mathbb{R}^N,
 \end{split}
 \end{equation}
where  $$F(w_j)=\mu_j^{-2s}\left[(U+\mu_j^{2s}w_j)^{2_s^*-1-\epsilon}-(2_s^{*}-1) \mu_j^{2s}U^{2_s^*-2}w_j-U^{2^*_s-1}\right].$$

\noindent Define the operator $L$ as follows:
$$L:=
(-\Delta)^s  -(2_s^{*}-1) U^{2_s^*-2}.$$
Then (\ref{res1+})  can be rewritten  as $$L w_j+V(x_j+\mu_j x)v_j=F(w_j)\ .$$

\begin{proposition} \label{PR2} Assume $N>6s$. Then $w_j\rightarrow w$ in $L^\infty(\mathbb{R}^N)$ as $j\rightarrow \infty$, where $w$ is a bounded solution of
\begin{equation}\label{res1+1}
   \begin{split}
 (-\Delta)^s w -(2_s^{*}-1) U^{2_s^*-2}w +V(x_0)U+ \widetilde{C}(N,s)U^{2^*_s-1} \ln U=0 \ \ \text{in}\ \ \mathbb{R}^N.
 \end{split}
 \end{equation}
\end{proposition}

\noindent In order to prove Proposition \ref{PR2}, we need the following result from \cite{JMY13} 

\begin{lemma}[Nondegeneracy]\label{JMY13}  The solution $U$ is nondegenerate in the sense that all bounded solutions of equation
$$(-\Delta)^s\phi=(2_s^*-1)U^{2_s^*-2}\phi\ \ \ \text{in}\ \ \mathbb{R}^N$$
are linear combinations of the functions
$$\frac{N-2s}{2}U+x\cdot \nabla U,\ \ \frac{\partial U}{\partial x_i}, \ i=1,2,\cdots,N.$$

\end{lemma}

\noindent Let $$X=\text{span}\left\{\frac{\partial U}{\partial x_1}, \cdots, \frac{\partial U}{\partial x_N}, \frac{N-2s}{2}U+x\cdot \nabla U \right\}$$
Clearly, $X\subset L^p(\mathbb{R}^N)$ with $\frac{N}{N-2s}<p<+\infty$.
For $1<r<\frac{N}{2s}$, define
$$Y_r:=\left\{u\in L^r(\mathbb{R}^N):\int_{\mathbb{R}^N}uvdx=0\ \ \text{for all}\ \ v\in X\right\}\ .$$
Then $L^r(\mathbb{R}^N)=X\oplus Y_q,$ where $\frac{N}{N-2s}<r<\frac{N}{2s}$.

\begin{lemma}\label{Lm21}
Suppose $N>4s$. Then for any  $1<q<\frac{N}{4s}$, there exists a constant $C>0$ such that
  \begin{equation}\label{Le21+1}\|u\|_{ \mathcal{W}^{2s,r}}\leq C (\|Lu\|_{r}+\|Lu\|_{q}), \end{equation}
for all $u\in Y_r\ \cap   \mathcal{W}^{2s,r}(\mathbb{R}^N) \cap C^2(\mathbb{R}^N)  $  with $Lu\in L^q(\mathbb{R}^N)$,  $\frac{1}{q}-\frac{2s}{N}=\frac{1}{r}$.
\end{lemma}

\begin{proof} It is enough to prove
$$\|u\|_{r}\leq C (\|Lu\|_{r}+\|Lu\|_{q}).$$
In fact, by
$$ (-\Delta)^s u +u=Lu+[1-(2_s^{*}-1) U^{2_s^*-2}]u $$
we get
$$\|u\|_{ \mathcal{W}^{2s,r}}\leq \|Lu\|_r+C\|u\|_r\leq C (\|Lu\|_{r}+\|Lu\|_{q}).$$

\noindent Assume that $u\not=0$. Otherwise, we are done.  By homogeneity, we can replace $u$ by $\frac{u}{\max\{\|u\|_{C^2},\|u\|_r\} }$ in (\ref{Le21+1}). Thus,
assume that there exists a sequence $\{u_n\}\subset Y_r\ \cap \mathcal{W}^{2s,r}(\mathbb{R}^N) \cap  C^2(\mathbb{R}^N) $   such that
   \begin{equation}\label{Lo9+0}
   \begin{split} \text{either} \ \ \|u_n\|_{C^2}=1,\ \ \|u_n\|_r<1, \ \ \text{or}\ \ \|u_n\|_{C^2}<1,\ \ \|u_n\|_r=1,
 \end{split}
 \end{equation}and
 \begin{equation}\label{Lo9+1}
   \begin{split}   \|Lu_n\|_{q}+ \|Lu_n\|_{r}\rightarrow 0.
 \end{split}
 \end{equation}
   Then, there exists $u_\infty\in  C^2(\mathbb{R}^N)$ such that after passing to a subsequence if necessary, $u_n\rightarrow u_\infty$ in $C_{loc}^2(\mathbb{R}^N)$ and in particular,  $u_n\rightarrow u_\infty$ in $L_{loc}^t(\mathbb{R}^N)$, $r\leq t<2^*_s$.   Let $I=(-\Delta)^{-s}$   the Riesz potentials defined by
   $$(I\ast f)(x)=\frac{1}{\gamma(s)}\int_{\mathbb{R}^N}\frac{f(y)}{|x-y|^{N-2s}}dy$$ with
   $$\gamma(s)=\pi^{\frac{N}{2}} 2^{2s}\Gamma(s)/\Gamma(N/2-s).$$  See Chapter V in \cite{Stein70}.
Then,  we have
$$u_n-I \ast [(2_s^*-1) U^{2_s^*-2} u_n]=I \ast Lu_n.$$
By Hardy-Littlewood-Sobolev inequality \cite{Lib83, Stein70}, we have
$$\|I\ast Lu_n \|_{r}\leq C\|Lu_n\|_{q}\rightarrow 0 $$
and H\"{o}lder's inequality yields
\begin{equation}\label{Lo9}
   \begin{split}
  & \|I \ast [(2_s^*-1) U^{2_s^*-2} (u_n-u_m)]\|_{r} \\
  &=(2_s^*-1)\left[\int_{B_R(0)}|I\ast U^{2_s^*-2}(u_n-u_m)|^{r}dx+ \int_{B_R^c(0)}|I\ast U^{2_s^*-2}(u_n-u_m)|^{r}dx\right]^{\frac{1}{r}}\\
  &\leq C \|U^{2_s^*-2}(u_n-u_m)\|_{L^{q}({B_R(0)})}+C\|u_n-u_m\|_{r}\left(\int_{|x|\geq R}U^{\frac{2N}{N-2s}}dx\right)^{\frac{2s}{N}}\\
  &\leq C \| u_n-u_m \|_{L^{r}({B_R(0)})}+C\|u_n-u_m\|_{r}\left(\int_{|x|\geq R}U^{\frac{2N}{N-2s}}dx\right)^{\frac{2s}{N}}\ ,
 \end{split}
 \end{equation}
where  $\frac{1}{q}-\frac{2s}{N}=\frac{1}{r}$.

\noindent Thus, $\{I \ast [(2_s^*-1) U^{2_s^*-2} u_n\}$ is a Cauchy sequence in $L^{r}(\mathbb{R}^N)$ and then $\{u_n\}$ is a Cauchy sequence in $L^{r}(\mathbb{R}^N)$. So, $ u_\infty \in L^r(\mathbb{R}^N)$, $u_\infty\in Y_r$ and    \begin{equation}\label{Lo9+}(-\Delta)^s u_\infty-(2_s^*-1)U^{2_s^*-2} u_\infty=0  \ \ \ \text{in}\ \ \mathbb{R}^N.\end{equation} By (\ref{Lo9+0}),    $u_\infty \in L^\infty(\mathbb{R}^N)$ and   $u_\infty\in X$ by (\ref{Lo9+}). But since $u_\infty\in Y_{r}$, we get $u_\infty\equiv 0$, which is a contradiction from (\ref{Lo9+0}).

\end{proof}

\noindent For fixed $j=1,2,\cdots$, we have 
\begin{equation}\label{res1+1}
   \begin{split}
 (-\Delta)^s v_j +v_j=(1+\mu_j^{2s}V(x_j+\mu_j x))v_j+v_j^{2_s^*-1-\epsilon_j} \ \ \text{in} \ \ \mathbb{R}^N.
 \end{split}
 \end{equation}
Note that $0\leq v_j\leq 1$. Thus,  $v_j\in \mathcal{W}^{2s,p}(\mathbb{R}^N) \cap C^{2,\beta}(\mathbb{R}^N)$  for   $p\in [2,+\infty) $ and $w_j\in   \mathcal{W}^{2s,p}(\mathbb{R}^N) \cap C^{2,\beta}(\mathbb{R}^N)$ for   $2\leq p<+\infty$.

\noindent Let
$$w_j=\sum_{i=1}^{N+1} a_{ij}e_i+z_j,\ \ j=1,2,\cdots,$$ where
$e_1=\frac{\partial U}{\partial x_i}$, $i=1,2,\cdots,N$, $e_{N+1}=\frac{N-2s}{2}U+x\cdot \nabla U,$ $z_j\in Y_q\cap C^{2,\beta}(\mathbb{R}^N),$ $\frac{1}{q}+\frac{1}{p}=1.$

\begin{lemma}  \label{Lm22} Assume $N>6s$ and
let $M_j=\max\{|a_{1j}|,|a_{2j}|,\cdots, |a_{(N+1)j}|\}$. Then
$M_j $ and $\|z_j\|_{\mathcal{W}^{2s,r}}$ are bounded as $j\rightarrow\infty$.

\end{lemma}

\begin{proof}
We may  assume without loss of generality, $M_j\rightarrow +\infty$ as $j\rightarrow \infty$ and
 $$\frac{1}{M_j}(a_{1j},\cdots, a_{(N+1)j})\rightarrow (b_1,\cdots,b_{N+1})\not=0,\ \ \text{as } \ j\rightarrow \infty.$$

 \noindent Then   \begin{equation}\label{Lo8-1}
   \begin{split}
  (-\Delta)^s \frac{z_j}{M_j}=(2_s^{*}-1) U^{2_s^*-2}\frac{z_j}{M_j}+\frac{1}{M_j}\left[F(w_j)-V(x_j+\mu_j x)v_j\right].
 \end{split}
 \end{equation}

\noindent Let us now  now estimate the three terms in the right hand side of equation (\ref{Lo8-1}). We have  \begin{equation}\label{Lo9-1}
   \begin{split}
|\mu_j^{2s}F(w_j)|=&|(U+\mu_j^{2s}w_j)^{2_s^*-1-\epsilon}-(2_s^{*}-1)\mu_j^{2s} U^{2_s^*-2}w_j-U^{2^*_s-1}|\\
\leq& |U^{2^*_s-1}-U^{2^*_s-1-\epsilon_j}|+| (2_s^{*}-1)\mu_j^{2s} U^{2_s^*-2}w_j-(2_s^{*}-1-\epsilon_j)\mu_j^{2s} U^{2_s^*-2-\epsilon_j}w_j|\\
&+|(U+\mu_j^{2s}w_j)^{2_s^*-1-\epsilon_j}-U^{2^*_s-1-\epsilon_j}-(2_s^{*}-1-\epsilon_j)\mu_j^{2s} U^{2_s^*-2-\epsilon_j}w_j|\\
=&I_1+I_2+I_3.
 \end{split}
 \end{equation}

\noindent Hence \begin{equation}\label{Lo9+1}
   \begin{split}
I_1=U^{2^*_s-1-\epsilon_j}|U^{\epsilon_j}-1|=U^{2^*_s-1-\epsilon_j}|\epsilon_j\ln U+o(\epsilon_j)|\leq \epsilon_j U^{2^*_s-1-\epsilon_j} (|\ln U|+1)
 \end{split}
 \end{equation} and since  $v_j\leq CU$ and $\mu_j^{2s}w_j=v_j-U$, we get
 \begin{equation}\label{Lo9+2}
   \begin{split}
I_2
=&\mu_j^{2s}U^{2_s^*-2-\epsilon_j}|w_j| | (2_s^{*}-1) U^{\epsilon_j} -(2_s^{*}-1-\epsilon_j) |\\
\leq &C\epsilon_j \mu_j^{2s}|w_j|U^{2_s^*-2-\epsilon_j}(|\ln U|+1)\\
\leq &C\epsilon_j  U^{2_s^*-1-\epsilon_j}(|\ln U|+1).
 \end{split}
 \end{equation}

\noindent Set 
$g(t)=(U+t\mu_j^{2s}w_j)^{2_s^*-1-\epsilon_j}$. Then,  we obtain  \begin{equation}\label{Lo9+3}
   \begin{split}
I_3
=& |g(1)-g(0)-g'(0)|\\
\leq&\int_0^1t|g''(1-t)|dt\\
\leq& C\int_0^1t[U+(1-t)\mu_j^{2s}w_j]^{2_s^*-3-\epsilon_j}\mu_j^{4s}w_j^2 dt\\
\leq& C \mu_j^{2s}|w_j||v_j-U|\int_0^1t U^{2_s^*-3-\epsilon_j} dt\\
\leq& C \mu_j^{2s}|w_j||v_j-U| U^{2_s^*-3-\epsilon_j}.
 \end{split}
 \end{equation}
Thus,  by (\ref{Lo9-1})--(\ref{Lo9+3}) and  Remark \ref{REM1}, we get
\begin{equation}\label{Lo12}
   \begin{split}
|F(w_j)|\leq C\left[U^{2^*_s-1-\epsilon_j} (|\ln U|+1)+|w_j||v_j-U| U^{2_s^*-3-\epsilon_j}\right].
 \end{split}
 \end{equation}
So, by dominated convergence we obtain 
\begin{equation}\label{Lo10}
   \begin{split}
\|F(w_j)\|_q\leq& C\left[\|U^{2^*_s-1-\epsilon_j} (|\ln U|+1)\|_q+\|w_j|v_j-U| U^{2_s^*-3-\epsilon_j}\|_q\right]\\
\leq& C\left[\|U^{2^*_s-1-\epsilon_j} (|\ln U|+1)\|_q+\|w_j\|_r\||v_j-U| U^{2_s^*-3-\epsilon_j}\|_{\frac{N}{2s}}\right]\\
\leq& C\left[1+o(1)\|w_j\|_r \right].
 \end{split}
 \end{equation}
  Thus, we get
\begin{equation}\label{Lo11}
   \begin{split}
\frac{1}{M_j}\|F(w_j)\|_q\leq  C\left[o(1)+o(1)\left\|\frac{z_j}{M_j}\right\|_r \right].
 \end{split}
 \end{equation}
Again by dominated convergence we get 
\begin{equation}\label{Lo12}
   \begin{split}
\|F(w_j)\|_r\leq& C\left[\|U^{2^*_s-1-\epsilon_j} (|\ln U|+1)\|_r+\|w_j|v_j-U| U^{2_s^*-3-\epsilon_j}\|_r\right]\\
\leq& C\left[1+o(1)\|w_j\|_r \right],
 \end{split}
 \end{equation}
which yields
\begin{equation}\label{Lo13}
   \begin{split}
\frac{1}{M_j}\|F(w_j)\|_r\leq  C\left[o(1)+o(1)\left\|\frac{z_j}{M_j}\right\|_r \right].
 \end{split}
 \end{equation}
By Lemma \ref{Lm21}, for $ \frac{N}{N-2s}<q<\frac{N}{4s}$ with $\frac{1}{q}-\frac{2s}{N}=\frac{1}{r}$ (Note that $N> 6s$ is needed),  we get
\begin{equation}\label{Lo14}
   \begin{split}
\left\|\frac{z_j}{M_j}\right\|_{\mathcal{W}^{2s,r}} \leq&  \frac{C}{M_j}\Big[\|V_j v_j\|_q+\|V_j v_j\|_r+\|F(w_j)\|_q+\|F(w_j)\|_r  \Big]\\
 \leq&   \frac{C}{M_j}\left[1+\|F(w_j)\|_q+\|F(w_j)\|_r  \right]\\
  \leq&   C\left[o(1)+o(1)\left\|\frac{z_j}{M_j}\right\|_r \right].
 \end{split}
 \end{equation}
Thus, we have
\begin{equation}\label{Lo15}
   \begin{split}
\left\|\frac{z_j}{M_j}\right\|_{\mathcal{W}^{2s,r}}=o(1).
 \end{split}
 \end{equation}
By Proposition \ref{lm21}, we get
\begin{equation}\label{Lo15}
   \begin{split}
\left\|\frac{z_j}{M_j}\right\|_{t}=o(1), \ \ r\leq t\leq \frac{Nr}{N-2sr}.
 \end{split}
 \end{equation}
By choosing $r$ close to $\frac{N}{2s}$,  $t$ can be arbitrarily large.
Besides, from  (\ref{Lo12}),  we have
\begin{equation}\label{Lo16-}
   \begin{split}
\left|L\left(\frac{z_j}{M_j}\right)\right|\leq&   C\frac{1}{M_j}\left[U+U^{2^*_s-1-\epsilon_j} (|\ln U|+1)+|w_j||v_j-U| U^{2_s^*-2-\epsilon_j}\right]\\
\leq&  o(1)\left[U+U^{2^*_s-1-\epsilon_j} (|\ln U|+1)+  U^{2_s^*-2-\epsilon_j}\sum_{i=1}^{N+1}|e_i|\right]+o(1)\left|\frac{z_j}{M_j}\right|,
 \end{split}
 \end{equation}
which yields that $L\left(\frac{z_j}{M_j}\right) \in L^t(\mathbb{R}^N)$. Thus, from (\ref{Lo15}),  we get
\begin{equation}\label{Lo17}
   \begin{split}
\left\|\frac{z_j}{M_j}\right\|_{\mathcal{W}^{2s,t}}=o(1).
 \end{split}
 \end{equation}
By Lemma \ref{lm21}, we have
\begin{equation}\label{Lo18}
   \begin{split}
\left\|\frac{z_j}{M_j}\right\|_{C^{0,\mu}}=o(1),
 \end{split}
 \end{equation}
for some $0<\mu<1$.
In particular we have $
\left\|\frac{z_j}{M_j}\right\|_{\infty}\leq C
 $ and from  (\ref{Lo16-}),  $
\left\|(-\Delta)^s\frac{z_j}{M_j}\right\|_{\infty}\leq C
 $. From Lemma 4.4 in \cite{CASI14}, $\left\|\frac{z_j}{M_j}\right\|_{C^{2,\beta}}\leq C
 $.
So,  $\frac{z_j}{M_j}\rightarrow 0$ in $L^\infty(\mathbb{R}^N)$  and $C_{loc}^1(\mathbb{R}^N)$ as $j\rightarrow \infty$. Since $v_j(0)=U(0)=1$ and both they achieve their maximum at 0, we get
\begin{equation}\label{Lo19}
   \begin{split}
 0=w_j(0)=M_j\left(\sum_{i=1}^{N+1}b_i e_i(0)+o(1)\right),\\
 0=\nabla w_j(0)=M_j\left(\sum_{i=1}^{N+1}b_i \nabla e_i(0)+o(1)\right).
 \end{split}
 \end{equation}
By direct calculations, it follows $(b_1, b_2,\cdots, b_{N+1})=0$, which is a contradiction.

\noindent Similarly, we can prove the remaining part of the Lemma.
\end{proof}

\begin{lemma} Assume $N>6s$. Then $z_j\rightarrow z$ in $C_{loc}^{1}(\mathbb{R}^N)$, where $z$ is radial and satisfies
\begin{equation}\label{Lo19-}
   \begin{split}
 (-\Delta)^s z-(2_s^{*}-1) U^{2_s^*-2} z+V(x_0)U- \widetilde{C}(N,s)U^{2^*_s-1} \ln U=0\ \ \text{in}\ \ \mathbb{R}^N.
 \end{split}
 \end{equation}
\end{lemma}

\begin{proof}
By  Lemma \ref{Lm22}, there exists a subsequence $\{z_{jk}\}$ such that
$z_{jk}\rightharpoonup z$ in $\mathcal{W}^{2s,r}$ and $z_{jk}\rightarrow z$ in $C_{loc}^{1}(\mathbb{R}^N)$, see also \cite{CASI14}.  Since $\|z_j\|_\infty$ is bounded, from (\ref{Lo9+2}) and  (\ref{Lo9+3}),   we get
\begin{equation}\label{Lo17}\frac{I_2+I_3}{\mu_j^{2s}}=o(1)  \end{equation} and
\begin{equation}\label{Lo18}   \begin{split} \frac{1}{\mu_j^{2s}}(U^{2^*_s-1}-U^{2^*_s-1-\epsilon_j})= &\frac{\epsilon_j\ln U+o(\epsilon_j)}{\mu_j^{2s}}U^{2^*_s-1}\\ =& \frac{ \mu_j^{2s}\widetilde{C}(N,s)\ln U+o(\mu_j^{2s})}{\mu_j^{2s}}U^{2^*_s-1}\\ =& \widetilde{C}(N,s)U^{2^*_s-1} \ln U +o(1).  \end{split}\end{equation} Thus, $z$ satisfies (\ref{Lo19-}).

\noindent Since $z_{jk}\in Y_r$, we get $z\in Y_r$. Thus, (\ref{Lo19-}) has at most one such solution, and $z_{j}\rightharpoonup z$ in $\mathcal{W}^{2s,r}$.
Moreover, since $(-\Delta)^s$ is invariant with respect to
the action of the orthogonal group $O(n)$ on $\mathbb{R}^N$ (see \cite{DDCA19}), if $T$ denotes a rotation in $\mathbb{R}^N$, since (\ref{Lo19-}) is invariant under  rotation,  then $z(Tx)-z(x)\in X$. Consequently, $z(Tx)=z(x)$. This proves that $z$ is radial.

\end{proof}
\begin{lemma} Assume $N>6s$. Then $ |a_{ij}| \rightarrow 0, $ $i=1,2,\cdots, N$ and $a_{(N+1)j}\rightarrow -\frac{2}{N-2s}z(0)$ as $j\rightarrow\infty$.
\end{lemma}

\begin{proof}
Note that \begin{equation}\label{Lo16}
   \begin{split}
 0=  \sum_{i=1}^{N+1}a_{ij} e_i(0)+z_j(0) ,\\
 0= \sum_{i=1}^{N+1}a_{ij} \nabla e_i(0)+\nabla z_j(0),
 \end{split}
 \end{equation}
which gives
 \begin{equation}\label{Lo16}
   \begin{split}
 0= \frac{N-2s}{2}a_{(N+1)j}+z_j(0) ,\\
 0= \sum_{i=1}^{N}b_i \nabla e_i(0)+\nabla z_j(0).
 \end{split}
 \end{equation}
Since $\nabla z(0)=0$, we get the result.

\end{proof}

\begin{lemma} Assume $N>6s$. Then $ w_j \rightarrow w$ in $L^\infty(\mathbb{R}^N)$ as $j\rightarrow \infty$, where
$$w=z- \frac{2}{N-2s}z(0)\left(\frac{N-2s}{2}U+x\cdot \nabla U\right). $$
\end{lemma}

\begin{proof} It sufficient to prove  $z_j\rightarrow z$ in $L^\infty(\mathbb{R}^N)$ as $j\rightarrow \infty$.
In fact, by consider $L(z_j-z)$, the proof is analogous to the proof of Lemma  \ref{Lm22}.

\end{proof}

\begin{theorem} \label{th6} Assume $N>6s$, $u_{\epsilon_j}$ is a ground state of (\ref{maineq0}) satisfying (\ref{fun00}) which has a maximum point $x_{\epsilon_j}$ satisfying $x_{\epsilon_j}\rightarrow x_0$ as $j \rightarrow \infty$. Then
 \begin{equation}\label{Lo30}
   \begin{split}
   S_{2^*_s-\epsilon_j} ^V=&    S + S^{-\frac{N-2s}{2s}}\mu_j^{2s}\int_{\mathbb{R}^n}\left[\frac{2}{2_s^*}\widetilde{C}_{N,s} U^{2_s^*}\ln U+  V(x_0)U^2 \right]dx\\
   &-\mu_j^{2s}\widetilde{C}_{N,s}\frac{2}{(2_s^*)^2}  S  \ln S^{\frac{N}{2s}} +o(\mu_j^{2s}). \end{split}
 \end{equation}

\end{theorem}

\begin{proof}

By the very definition of $  S_{2^*_s-\epsilon_j}^V$, we have
    \begin{equation}\label{Lo7}
   \begin{split}
   \left(S_{2^*_s-\epsilon_j}^V\right)^{\frac{2_s^*-\epsilon_j}{2_s^*-2-\epsilon_j}}=&  \int_{\mathbb{R}^n}v_j^{2_s^*-\epsilon_j}dx \\
   =& \int_{\mathbb{R}^n}\left(U+\mu_j^{2s} w_j\right)^{2_s^*-\epsilon_j}dx \\
   =& \int_{\mathbb{R}^n}\Big(U^{2_s^*-\epsilon_j}+(2_s^*-\epsilon_j)U^{2_s^*-1-\epsilon_j} \mu_j^{2s}w_j
\\&+\frac{1}{2} (2_s^*-\epsilon_j)(2_s^*-1-\epsilon_j) (U+t\mu_j^{2s} w_j)^{2_s^*-2-\epsilon_j}\mu_j^{4s} w_j^2\Big)dx,\ \ t\in (0,1).
   \end{split}
 \end{equation}
Since $v_j\leq CU$ and $v_j=U+\mu_j^{2s}w_j$, then by Lemma \ref{conver2}, 
    \begin{equation}\label{Lo8}
   \begin{split}
   &\int_{\mathbb{R}^n}  (U+t\mu_j^{2s} w_j)^{2_s^*-2-\epsilon_j}\mu_j^{4s} w_j^2dx \\ &\leq C\|w_j\|_\infty \mu_j^{2s}\int_{\mathbb{R}^n}  U^{2_s^*-2-\epsilon_j}|v_j-U| dx\\
   &\leq  C\|w_j\|_\infty \mu_j^{2s} \left(\int_{\mathbb{R}^n}  U^{\frac{2_s^*(2_s^*-2-\epsilon_j)}{2_s^*-1}}  dx\right)^{\frac{2_s^*-1}{2_s^*}}
   \left(\int_{\mathbb{R}^n}   |v_j-U|^{2_s^*} dx\right)^{\frac{1}{2_s^*}}\\
  & =o(\mu_j^{2s}).
   \end{split}
 \end{equation}

By (\ref{Lo7}), we get
    \begin{equation}\label{Lo9}
   \begin{split}
   \left(S_{2^*_s-\epsilon_j}^V\right)^{\frac{2_s^*-\epsilon_j}{2_s^*-2-\epsilon_j}}= &
  \int_{\mathbb{R}^n}\Big[U^{2_s^*-\epsilon_j}+(2_s^*-\epsilon)U^{2_s^*-1-\epsilon_j} \mu_j^{2s}w_j\Big]dx+o(\mu_j^{2s})\\
  =& \int_{\mathbb{R}^n}\Big[U^{2_s^*}-\epsilon_j U^{2_s^*}\ln U+(2_s^*-\epsilon_j)U^{2_s^*-1-\epsilon_j} \mu_j^{2s}w_j\Big]dx+o(\epsilon_j)+o(\mu_j^{2s})\\
  =& \int_{\mathbb{R}^n}\Big[U^{2_s^*}-\epsilon_j U^{2_s^*}\ln U+2_s^*U^{2_s^*-1} \mu_j^{2s}w_j\Big]dx+o(\epsilon_j)+o(\mu_j^{2s})\\
  =&S^{\frac{N}{2s}}+\mu_j^{2s}\int_{\mathbb{R}^n}\left[-\widetilde{C}_{N,s} U^{2_s^*}\ln U+2_s^*U^{2_s^*-1}  w\right]dx+o(\mu_j^{2s})
  \end{split}
 \end{equation}

\noindent By
(\ref{res1+1}), we get
\begin{equation}\label{Lo09}
   \begin{split}
    \int_{\mathbb{R}^n}   U^{2_s^*-1} w dx =&\int_{\mathbb{R}^n} (-\Delta)^s U w dx\\
     =& \int_{\mathbb{R}^n} (-\Delta)^s w Udx \\
       =&\int_{\mathbb{R}^n}[(2_s^{*}-1) U^{2_s^*-2}w -V(x_0)U- \widetilde{C}_{N,s} U^{2^*_s-1}\ln U] U dx.
  \end{split}
 \end{equation}
Thus,
\begin{equation}\label{Lo09}
   \begin{split}
   (2_s^{*}-2)\int_{\mathbb{R}^n}   U^{2_s^*-1} w dx=\int_{\mathbb{R}^n}[V(x_0)U+ \widetilde{C}_{N,s}U^{2^*_s-1}\ln U] U dx.
  \end{split}
 \end{equation}
So,
    \begin{multline}\label{Lo9}
  \left( S_{2^*_s-\epsilon_j}^V\right)^{\frac{2_s^*-\epsilon_j}{2_s^*-2-\epsilon_j}}=   S^{\frac{N}{2s}}\\+\mu_j^{2s}\int_{\mathbb{R}^n}\left[-\widetilde{C}_{N,s} U^{2_s^*}\ln U+\frac{2_s^*}{2_s^*-2} \left(V(x_0)U^2+ \widetilde{C}_{N,s}U \ln U\right)\right]dx+o(\mu_j^{2s})\\
   =    S^{\frac{N}{2s}}+\mu_j^{2s}\int_{\mathbb{R}^n}\left[\frac{2}{2_s^*-2}\widetilde{C}_{N,s} U^{2_s^*}\ln U+\frac{2_s^*}{2_s^*-2}  V(x_0)U^2 \right]dx+o(\mu_j^{2s})\ .
 \end{multline}
Thus, we have
    \begin{equation}\label{Lo9}
   \begin{split}
   S_{2^*_s-\epsilon_j}^V =&    S +\frac{2s}{N}S^{-\frac{N-2s}{2s}}\mu_j^{2s}\int_{\mathbb{R}^n}\left[\frac{2}{2_s^*-2}\widetilde{C}_{N,s} U^{2_s^*}\ln U+\frac{2_s^*}{2_s^*-2}  V(x_0)U^2 \right]dx\\
   &-\epsilon_j\frac{2}{(2_s^*)^2}  S  \ln S^{\frac{N}{2s}} +o(\mu_j^{2s}), \end{split}
 \end{equation}
which yields 
   \begin{equation}\label{Lo29}
   \begin{split}
   S_{2^*_s-\epsilon_j}^V =&    S + S^{-\frac{N-2s}{2s}}\mu_j^{2s}\int_{\mathbb{R}^n}\left[\frac{2}{2_s^*}\widetilde{C}_{N,s} U^{2_s^*}\ln U+  V(x_0)U^2 \right]dx\\
   &-\mu_j^{2s}\widetilde{C}_{N,s}\frac{2}{(2_s^*)^2}  S  \ln S^{\frac{N}{2s}} +o(\mu_j^{2s}). \end{split}
 \end{equation}
The proof is complete.

\end{proof}

\begin{theorem}  \label{th7} Assume $(V_1),$ $(V_2)$ with  $\inf\limits_{x\in \mathbb{R}^N} V(x)<\sup\limits_{x\in \mathbb{R}^N}  V(x)$,   $N>6s$ and let $u_\epsilon$ be the ground state of (\ref{maineq0}) which has a maximum point at $x_\epsilon$. Then, up to a subsequence, $V(x_{\epsilon_j})\rightarrow \min_{x\in \mathbb{R^N}} V(x)$ as $\epsilon_j\rightarrow 0^+$.

\end{theorem}

\begin{proof} By Theorems \ref{th5} and  \ref{th6}, it is sufficient to prove    that $x_\epsilon$ remains bounded. We argue by contradiction. Assume that there exists a sequence $x_j\rightarrow \infty$ such that
     \begin{equation}\label{inf}
   \begin{split}
   \lim_{j\rightarrow \infty}V(x_j)=V_\infty>\inf\limits_{x\in \mathbb{R}^N} V(x).
   \end{split}
    \end{equation}
By Remark \ref{REM1},
$$\epsilon_j=A\mu_j^{2s}+o(\mu_j^{2s})$$ for some $A>0$. Analogous to the proof of Theorems \ref{th5} and  \ref{th6} with $\widetilde{C}_{N,s}$ and $V(x_0)$ replaced by $A$
and $V_\infty$, we get $V_\infty\leq \inf\limits_{x\in \mathbb{R}^N} V(x)$, which contradicts  (\ref{inf}).

\end{proof}

\noindent {\it Proof  of Theorem \ref{th3}}. It follows from Theorems  \ref{th2},  \ref{th5},  \ref{th6}  and \ref{th7}.

\section{Local uniqueness: proof  of Theorem \ref{th4}} \label{S4}

\noindent Let us  argue by contradiction.  Suppose that there exists a sequence $\epsilon_j\rightarrow 0$ and two ground  states  far apart, namely $u_j^1:=u_{\epsilon_j}^1$ and $u_j^2:=u_{\epsilon_j}^2$.
Set
$$v_j^i(x):=(\mu_j^i)^{\frac{2s}{2_s^*-2-\epsilon_j}}u_j^i(\mu_j^i x),\ \ i=1,2.$$
Then $v_j^i\rightarrow U$ in $C_{loc}^{2,\beta}(\mathbb{R}^N)$ for $i=1,2$ as $j\rightarrow \infty.$

\noindent Assume further that $v_j^1\not=v_j^2$.
Set $$\theta_j:=v_j^1-v_j^2,\ \ \psi_j:=\frac{v_j^1-v_j^2}{\|v_j^1-v_j^2\|_\infty}.$$ Then
\begin{equation}\label{un1}
   \begin{split}
 (-\Delta)^s \psi_j+(\mu_j^1)^{2s}V(\mu_j^1x)v_j^1-(\mu_j^2)^{2s}V(\mu_j^2 x)v_j^2=\Phi_n\psi_n \ \ \text{in} \ \ \mathbb{R}^N,
 \end{split}
 \end{equation}
where
\begin{equation}\label{un2}
   \begin{split}
 \Phi_n=(2_s^*-1-\epsilon_j)\int_0^1\Big[tv_j^1(x)+(1-t)v_j^2(x)\Big]^{2_s^*-2-\epsilon_j}dt.
   \end{split}
 \end{equation}
Since $\|v_j^i\|_\infty=1$, $i=1,2$, by standard regularity  we have
 $\psi_j\rightarrow \psi$ in $C_{loc}^{2,\beta}(\mathbb{R}^N)$.  By Lemma  \ref{conver2}, we have that $\{\psi_j\}$ is uniformly
 bounded in $H_{V}^s(\mathbb{R}^N)$. Without loss of generality, we may assume that  $\psi_n\rightharpoonup \psi$ in $H_{V}^s(\mathbb{R}^N)$.

\noindent From (\ref{un1}), we have
\begin{equation}\label{un3}
   \begin{split}
 &\int_{\mathbb{R}^N}(-\Delta)^{\frac{s}{2}} \psi_j (-\Delta)^{\frac{s}{2}} \varphi dx \\
 &=-(\mu_j^1)^{2s} \int_{\mathbb{R}^N} V(\mu_j^1x)v_j^1\varphi dx
 +(\mu_j^2)^{2s} \int_{\mathbb{R}^N} V(\mu_j^2 x)v_j^2 \varphi dx\\& + \int_{\mathbb{R}^N}\Phi_n\psi_n \varphi dx,\  \ \forall \varphi \in C_0^\infty(\mathbb{R}^N).
 \end{split}
 \end{equation}
Taking $j\rightarrow \infty$ in (\ref{un3}), we get
\begin{equation}\label{un4}
   \begin{split}
 (-\Delta)^s \psi =(2_s^*-1)U^{2_s^*-2}\psi \ \ \text{in} \ \ \mathbb{R}^N.
 \end{split}
 \end{equation}
Note that  $\|\psi_j\|_\infty=1$ implies $\|\psi\|_\infty=1$. By Lemma \ref{JMY13}, $$\psi\in X=\text{span}\left\{\frac{\partial U}{\partial x_1}, \cdots, \frac{\partial U}{\partial x_N}, \frac{N-2s}{2}U+x\cdot \nabla U \right\}.$$
\noindent On the other hand, since  $v_j^i$ is radially symmetric,   $\psi$ is a radial function as well. Thus,   
$$\psi(x)=c\frac{\lambda^2-|x|^2}{(\lambda^2+|x|^2)^{\frac{N-2s+2}{2}}}$$
for some constant $c\in \mathbb{R}$.

\noindent We next actually prove that $c=0$. Indeed, otherwise assume for simplicity $c=1$.
By Pohozaev's identity, we have
\begin{equation}\label{un5}
   \begin{split}
& \frac{1}{N} (\mu_j^i)^{2s} \int_{\mathbb{R}^N} V(\mu_j^i x) |v_j^i|^2dx+\frac{1}{2N}(\mu_j^i)^{2s+1} \int_{\mathbb{R}^N}  x\cdot \nabla V(\mu_j^i x) |v_j^i|^2dx\\&= \frac{\epsilon_j}{2_s^*(2_s^*-\epsilon_j)} \int_{\mathbb{R}^N} |v_j^i|^{2_s^*-\epsilon_j}dx, \ \ i=1,2.
 \end{split}
 \end{equation}
By Remark \ref{REM1}, we get $(\mu_j^i)^{2s}\sim \epsilon_j$. Thus, from  (\ref{un5}), we have
\begin{equation}\label{un6}
   \begin{split}
 \frac{V_\infty}{N}   \int_{\mathbb{R}^N}\psi_j(   v_j^1+   v_j^2)dx+o(1)\geq  \frac{2_s^*-\epsilon_j}{2_s^*(2_s^*-\epsilon_j)} \int_{\mathbb{R}^N} \psi_j \int_0^1 \Big[ tv_j^1+(1-t)v_j^2 \Big]^{2_s^*-1-\epsilon_j}dtdx.
 \end{split}
 \end{equation}
Notice that
\begin{equation}\label{un6}
   \begin{split}
 \lim_{j\rightarrow \infty} \int_{\mathbb{R}^N}\psi_j(   v_j^1+   v_j^2)dx =2\int_{\mathbb{R}^N}\psi Udx=2\lambda^{N-2s}\int_{\mathbb{R}^N}\frac{\lambda^2-|x|^2}{(\lambda^2+|x|^2)^{ N-2s+1}}dx.
 \end{split}
 \end{equation}
Direct calculations show that
\begin{equation}\label{un7}
   \begin{split}
  \int_{\mathbb{R}^N}\frac{\lambda^2-|x|^2}{(\lambda^2+|x|^2)^{ N-2s+1}}dx<0.
 \end{split}
 \end{equation}
Thus, from (\ref{un6}),  we obtain
\begin{equation}\label{un8}
   \begin{split}
 \lim_{j\rightarrow \infty} \int_{\mathbb{R}^N}\psi_j(   v_j^1+   v_j^2)dx <0.
 \end{split}
 \end{equation}

\noindent On the other hand, we have
\begin{equation}\label{un9}
   \begin{split}
  \lim_{j\rightarrow \infty}  \int_{\mathbb{R}^N} \psi_j \int_0^1 \Big[ tv_j^1+(1-t)v_j^2 \Big]^{2_s^*-1-\epsilon_j}dtdx= \int_{\mathbb{R}^N} \psi U^{2_s^*-1}.
 \end{split}
 \end{equation}
By a suitable scaling we end up with 
\begin{equation}\label{un10}
   \begin{split}
  \int_{\mathbb{R}^N} \psi U^{2_s^*-1}\sim& \int_0^{+\infty}\frac{(1-r^2)r^{N-1}}{(1+r^2)^{1+N}}dr=0.
 \end{split}
 \end{equation}
By combining (\ref{un6})--(\ref{un10}),  we get a contradiction.

\noindent Thus, $c=0$ and $\psi_j\rightarrow 0$ in  $\Omega\subset\subset \mathbb{R}^N$.  If we let $y_j\in \mathbb{R}^N$ such that $\psi_j(y_j)=\|\psi_j\|_\infty=1$, then $y_j\rightarrow +\infty$ as $j\rightarrow \infty.$  However, by Lemma \ref{comp}, we get $v_j^i(x)\leq C\frac{1}{|x|^{N-2s}}$, $i=1,2$  and thus $|\psi_j(x)|\leq C\frac{1}{|x|^{N-2s}}$, which implies $\psi_j(x)\rightarrow 0$ as $|x|\rightarrow \infty$. This is a contradiction since  $\psi_j(y_j)=1$.

\noindent The proof of Theorem \ref{th4} is now complete.

\end{document}